\documentclass{amsart}
\usepackage{subcaption}
\usepackage{amssymb}
\usepackage{bm}
\usepackage{mathtools}
\usepackage{stmaryrd}
\usepackage{graphicx}
\usepackage[dvipsnames]{xcolor}
\usepackage[colorlinks=true,linkcolor=Sepia,citecolor=ForestGreen,urlcolor=BurntOrange]{hyperref}
\usepackage[capitalise,nameinlink,noabbrev]{cleveref}
\usepackage[mathcal]{euscript}

	\makeatletter
	\newcommand{\dummylabel}[2]{\def\@currentlabel{#2}\label{#1}}
	\makeatother

\renewcommand{\leq}{\leqslant}

\renewcommand{\epsilon}{\varepsilon}

\newcommand{\Opens}{\mathop{\mathcal{O}\!}}

\newcommand{\sqleq}{\sqsubseteq} 

\newcommand{\up}{\mathop{\uparrow}\mspace{-2mu}}
\newcommand{\down}{\mathop{\downarrow}\mspace{-2mu}}

\newcommand{\Up}{\mathop{\text{\scalebox{1.3}[1.2]{\rotatebox[origin=c]{90}{$\rightarrowtriangle$}}}}\!}
\newcommand{\Down}{\mathop{\text{\scalebox{1.3}[1.2]{\rotatebox[origin=c]{-90}{$\rightarrowtriangle$}}}}\!}
\newcommand{\interior}{^{\circ}}

\DeclareMathOperator{\Clopup}{Clop^{\up}}

\newcommand{\cat}[1]{\mathbf{#1}}

\newcommand{\Top}{\cat{Top}}
\newcommand{\Frame}{\cat{Frm}}
\newcommand{\Loc}{\cat{Loc}}

\newcommand{\Ord}{\cat{Ord}}
\newcommand{\OrdLoc}{\cat{OrdLoc}}
\newcommand{\OrdTop}{\cat{OrdTop}}
\newcommand{\OC}{\mathrm{OC}}
\newcommand{\Rel}{\cat{Rel}}
\newcommand{\Heyt}{\cat{Heyt}}

\newcommand{\coHeyt}{\cat{coHeyt}}
\newcommand{\biHeyt}{\cat{biHeyt}}
\newcommand{\Esakia}{\cat{Esak}}
\newcommand{\coEsakia}{\cat{coEsak}}
\newcommand{\biEsakia}{\cat{biEsak}}

\DeclareMathOperator{\pt}{pt}
\DeclareMathOperator{\pf}{pf}
\newcommand{\op}{^{\mathrm{op}}}
\newcommand{\id}[1][]{\mathrm{id}_{#1}}

\newcommand{\Leq}{\mathrel{\trianglelefteqslant}}
\newcommand{\LeqU}{\mathrel{\trianglelefteqslant_\mathrm{U}}}
\newcommand{\LeqL}{\mathrel{\trianglelefteqslant_\mathrm{L}}}
\newcommand{\LeqEM}{\mathrel{\trianglelefteqslant_\mathrm{EM}}}
\newcommand{\leqU}{\leq_\mathrm{U}}
\newcommand{\leqL}{\leq_\mathrm{L}}

\theoremstyle{plain}
\newtheorem{theorem}{Theorem}[section]   
\newtheorem*{theorem*}{Theorem}
\newtheorem{proposition}[theorem]{Proposition}
\newtheorem{corollary}[theorem]{Corollary}
\newtheorem{lemma}[theorem]{Lemma}

\theoremstyle{definition}
\newtheorem{definition}[theorem]{Definition}
\newtheorem{example}[theorem]{Example}

\usepackage{tikz,tikz-cd}
\usetikzlibrary{backgrounds}
\usetikzlibrary{arrows,arrows.meta}
\usetikzlibrary{shapes,shapes.geometric,shapes.misc}
\pgfdeclarelayer{edgelayer}
\pgfdeclarelayer{nodelayer}
\pgfsetlayers{background,edgelayer,nodelayer,main}
\tikzstyle{none}=[inner sep=0mm]
\tikzstyle{tikzfig}=[baseline=-0.25em,scale=0.5]
\tikzstyle{RRed_text}=[text={rgb,255: red,190; green,49; blue,26}]
\tikzstyle{dot_point}=[fill=black, draw=none, shape=circle, inner sep=1pt]
\tikzstyle{thick_dotted}=[-, dotted, thick]
\tikzstyle{thin_dotted}=[-, dotted]
\tikzstyle{green_open}=[-, fill={rgb,255: red,225; green,255; blue,235}, very thick]
\tikzstyle{green_cone}=[-, draw=none, fill={rgb,255: red,176; green,255; blue,176}, fill opacity=.5]
\tikzstyle{blue_cone}=[-, draw=none, fill={rgb,255: red,205; green,249; blue,255}, fill opacity=.34]
\tikzstyle{blue_open}=[-, fill={rgb,255: red,239; green,255; blue,255}, very thick]
\tikzstyle{yellow_opencone}=[fill={rgb,255: red,255; green,255; blue,220}, draw=none]
\tikzstyle{yellow_open}=[-, fill={rgb,255: red,255; green,255; blue,220}, fill opacity=1, very thick]
\tikzstyle{yellow_cone}=[-, draw=none, fill={rgb,255: red,255; green,244; blue,185}, fill opacity=.5]
\tikzstyle{yellow_cone_light}=[-, draw=none, fill={rgb,255: red,255; green,244; blue,185}, fill opacity=.2]
\tikzstyle{red_open}=[-, fill={rgb,255: red,255; green,228; blue,228}, very thick]
\tikzstyle{red_cone}=[-, fill={rgb,255: red,255; green,230; blue,230}, fill opacity=.5, draw=none]
\tikzstyle{red_dashed}=[-, fill={rgb,255: red,255; green,228; blue,228}, very thick, draw={rgb,255: red,71; green,26; blue,26}, dash pattern={{on 10pt off 2pt on 1pt off 2pt}}]
\tikzstyle{grey_open}=[-, fill={rgb,255: red,240; green,240; blue,240}, very thick]
\tikzstyle{grey_cone}=[-, draw=none, fill={rgb,255: red,218; green,218; blue,218}, fill opacity=.5]
\tikzstyle{real_line}=[-, ultra thick]
\tikzstyle{real_line_blue}=[ultra thick, draw={rgb,255: red,135; green,197; blue,255}, {|-|}]
\tikzstyle{real_line_red}=[ultra thick, draw={rgb,255: red,191; green,0; blue,3}, {|-|}]
\tikzstyle{real_line_yellow}=[ultra thick, draw={rgb,255: red,217; green,145; blue,0}, {|-|}]
\tikzstyle{real_line_grey}=[ultra thick, draw={rgb,255: red,171; green,171; blue,171}, {|-|}]
\tikzstyle{preimage_frames}=[->]

\title{Ordered Locales}
\author{Chris Heunen}
\address{University of Edinburgh}
\email{chris.heunen@ed.ac.uk}
\author{Nesta van der Schaaf}
\address{University of Edinburgh}
\email{n.schaaf@ed.ac.uk}

\date{\today}\thanks{We thank Prakash Panangaden for useful discussions, and Mai Gehrke, Sam van Gool, Robin Kaarsgaard, Luca Reggio, and Morgan Rogers for helpful suggestions.}

\begin{document}

\begin{abstract}
  We extend the Stone duality between topological spaces and locales to include order: there is an adjunction between the category of preordered topological spaces satisfying the so-called \emph{open cone} condition, and the newly defined category of \emph{ordered locales}. The adjunction restricts to an equivalence of categories between spatial ordered locales and sober $T_0$-ordered spaces with open cones.
\end{abstract}

\maketitle

\section{Introduction}

There are many reasons why one might want to equip (the underlying set of) a topological space with a preorder. 
In topology itself, ordered spaces were first used to study \emph{fixed point theorems}~\cite{ward:orderedspaces} and \emph{compactifications}~\cite{nachbin}.
But they also give \emph{algebraic} means to study topological spaces~\cite{priestley:lattices}, providing models for \emph{(modal) logic}.
This is especially useful in \emph{computability-} and \emph{domain theory}~\cite{gierzetal:domains}, where the order indicates levels of knowledge about an ongoing computation.
Similarly, ordered spaces are used in the study of \emph{concurrent computing}~\cite{haucourt2009ComparingTopologicalModels,fajstrup2016DirectedAlgebraicTopology}, where topological invariants can recognise phenomena such as deadlock.
A related use of ordered spaces is in \emph{directed algebraic topology}~\cite{grandis2009DirectedAlgebraicTopology}, an abstract setting for homotopy theory where paths need not be reversible.
Finally, ordered spaces have been used in the physics of \emph{spacetime}~\cite{kronheimer1967structure} and \emph{quantum gravity}~\cite{sorkinetal:causalset}, where the order models whether there can be a \emph{causal} influence of one point on another~\cite{panangaden2014CausalityPhysicsComputation}.

There are also many reasons why one might want to work with \emph{locales} instead of topological spaces. 
Working with the family of open sets while disregarding any possible points leads to \emph{better behaved} spaces~\cite{johnstone1982StoneSpaces}, which is intimately related to \emph{constructive} proofs~\cite{johnstone:point}. 
It allows considering rigorously the topological intuition inherent to \emph{spaces without points}~\cite{picado2012frames} that occur naturally throughout mathematics.
This is most pronounced in \emph{topos theory}~\cite{johnstone2002elephant}; pragmatically, any theorem proved locale-theoretically automatically also applies to for examples \emph{sheaves}~\cite{maclane1994sheaves}. 
Similarly, \emph{synthetic topology}~\cite{escardo:synthetic} can extract computational algorithms from locale-theoretic continuity, and locale-theoretic methods make \emph{logical} methods available to reason about topology~\cite{vickers1996topology}.
Finally, there are unresolved foundational discussions about the nature of physical spacetime, which may not be continuous~\cite{sorkinetal:causalset} or may not have an empirically accessible notion of point \cite{forrest1996OntologyTopologyTheory}.

\emph{Stone duality} famously links Boolean algebras and \emph{Stone spaces,} a special type of topological space \cite{stone1936TheoryRepresentationBoolean}. Since the publication of that seminal paper, many related and generalised results have been published. Most relevant for us is the Stone duality between topological spaces and locales, giving an adjunction between the two categories~\cite{johnstone1982StoneSpaces}: given a topological space, its partially ordered family of open sets forms a locale, and conversely, given a locale, one can consider the topological space of its points, suitably defined.
In this article, we extend this duality to take preorders into account. 
To what structure on its locale does a preorder on a topological space correspond?

Our answer is an axiomatisation of the appropriate notion of \emph{ordered locale}. 
We lift the preorder $x \leq y$ on points of a topological spaces to a preorder $U \Leq V$ on the open subsets of the topological space (to be precise, the lower/upper/Egli-Milner preorder), and axiomatise the latter.
For a working theory, we restrict to preorders on topological spaces satisfying the \emph{open cone} condition: the up- and downsets of an open set are again open. 
We establish an adjunction between the following two categories: 
preordered topological spaces with open cones and enough points, and continuous monotone functions; 
and ordered locales and monotone locale morphisms.
The adjunction restricts to an equivalence between the full subcategories of sober $T_0$-ordered spaces with open cones and spatial ordered locales.
This could also be regarded as a natural variation on the study of Boolean algebras with additional operations~\cite{jonssontarski}. 

It should be emphasised that we study ordered locales from the perspective of causality in a point-free setting, rather than from the perspective of duality theory. There, the motivation is often to find a topological representation of some algebraic  or lattice-theoretic structure:
\begin{itemize}
	\item \emph{Stone duality:} Boolean algebras as Stone spaces \cite{stone1936TheoryRepresentationBoolean};
	\item \emph{Priestley duality:} bounded distributive lattices as Priestley spaces \cite{priestley:lattices};
	\item \emph{Esakia duality:} Heyting algebras as Esakia spaces \cite{esakia1974topologicalKripkemodels}.
\end{itemize}
Our motivation, however, is in generalising the notion of ordered topological space into the point-free realm. Our adjunction is a means to justify the definition of an ordered locale. That said, we do exhibit a relationship between our result and \emph{Esakia duality}.

The outline of the article is as follows. \Cref{section:ordered topological spaces} starts by discussing ordered spaces, and \Cref{section:ordered locales} introduces ordered locales and their morphisms. Then, \Cref{sec:opens} studies the functor that takes opens, turning an ordered space into an ordered locale. Vice versa, \Cref{sec:points} studies the functor that takes points, turning an ordered locale into an ordered space. Putting it all together, \Cref{sec:adjunction} proves the adjunction, and investigates its fixed points to establish Stone duality in the ordered setting. \Cref{section:upper-lower adjunctions,section:esakia duality} discuss slightly generalised results and relations to Esakia duality. Finally, \Cref{sec:conclusion} concludes by raising directions for further research.

\section{Ordered spaces}\label{section:ordered topological spaces}

We start with basic properties of preordered sets.
A \emph{preorder}, which we also simply call \emph{order}, is a binary relation on a set that is reflexive ($x \leq x$) and transitive ($x \leq y \leq z \implies x \leq z$).
A function $g \colon S \to T$ between preordered sets is \emph{monotone} if $x \leq y$ implies $g(x) \leq g(y)$.
Preordered sets and monotone functions form a category, $\Ord$.
If $(S,\leq)$ is a preordered set, we write
\[
  \up A = \{ y \in S \mid \exists x \in A \colon x \leq y \}
  \qquad
  \down A = \{ x \in S \mid \exists y \in A \colon x \leq y \}
\]
for the \emph{up-} and \emph{downset} of a subset $A \subseteq S$. Thinking of $\leq$ as a causality relation, we may also call these \emph{future-} and \emph{past cones}.

\begin{lemma}\label{lem:cones}
  If $(S, \leq)$ is a preordered set, then:
	\begin{enumerate}
		\item[(a)] $A\subseteq \up A$ and $A \subseteq \down A$ for any subset $A$;
		\item[(b)] $\up A = \up \up A$ and $\down A = \down \down A$ for any subset $A$;
		\item[(c)] $\up A \subseteq \up B$ and $\down A \subseteq \down B$ for any subsets $A \subseteq B$;
		\item[(d)] $\bigcup \up A_i = \up \Big(\bigcup A_i\Big)$ and $\bigcup \down A_i = \down \Big(\bigcup A_i\Big)$ for any family $A_i$ of subsets.
	\end{enumerate}
\end{lemma}
\begin{proof}
  The first property follows from reflexivity, the second from transitivity, and the third holds by construction.	For the fourth property, note that $\up \varnothing = \varnothing = \down \varnothing$, so we may assume the indexing set is not empty. 
  Then $A_j \subseteq \bigcup A_i$ for any index $j$, and so $\bigcup \up(A_i) \subseteq \up\big(\bigcup A_i \big)$ by the third property. For the converse inclusion, pick $y \in \up \big( \bigcup A_i \big)$. Then there exist an index $j$ and an element $x \in A_j$ such that $x \leq y$. But this just means that $y \in \up A_j$, and by extension $y \in \bigcup \up A_i$. The argument for downsets is analogous.
\end{proof}

Next we consider how inverse images of monotone functions interact with cones.

\begin{proposition}\label{proposition:monotone iff commutes with cones}
  For a function $g \colon S \to T$ between preordered sets, the following are equivalent:
	\begin{enumerate}
		\item[(a)] $g$ is monotone;
		\item[(b)] $\up g^{-1}(B) \subseteq g^{-1}(\up B)$ for all $B \subseteq T$;
		\item[(c)] $\down g^{-1}(B) \subseteq g^{-1}(\down B)$ for all $B \subseteq T$.
	\end{enumerate}
\end{proposition}		
\begin{proof}
	First assume (a); we will show (b). 
	Take $y\in \up g^{-1}(B)$. Then there exists $x\in g^{-1}(B)$ with $x \leq y$. Monotonicity of $g$ gives $g(x)\leq g(y)$. But $g(x)\in B$, so $g(y)\in \up B$, and therefore $y\in g^{-1}(\up B)$.

	Next, assume (b); we will show (a). 
	Suppose $x\leq y$ in $S$, that is, $y \in \up \{x\}$.
	Furthermore $x\in g^{-1}\{g(x)\}$. Hence, by assumption:
	\[
		y \in \up \{x\}
		\subseteq \up g^{-1}\{g(x)\}
		\subseteq g^{-1}(\up \{g(x)\}).
	\]
	But this just means $g(y)\in \up \{g(x)\}$, that is, $g(x)\leq g(y)$. 

	The equivalence between (a) and (c) is proved similarly.
\end{proof}

Now we move from ordered sets to ordered spaces.

\begin{definition}
	An \emph{ordered topological space}, or \emph{ordered space} for short, is a topological space equipped with a preorder on its underlying set. Ordered spaces and continuous monotone functions form a category $\OrdTop$.
\end{definition}

Note that this definition imposes \emph{no relation} between the topology and the order. This diverges from parts of the literature, notably~\cite{nachbin}, which use the same terminology to refer to spaces with order relations that are closed in the product topology. As this property is not important to us (see \cref{example:no open cones}), we adopt a more general definition.


\begin{definition}
	We say an ordered space $S$ has:
		\begin{itemize}
			\item \emph{open upper cones} if $\up U$ is open whenever $U\subseteq S$ is open;
			\item \emph{open lower cones} if $\down U$ is open whenever $U\subseteq S$ is open;
			\item \emph{open cones} if it has both open upper cones and open lower cones.
		\end{itemize}
	Write $\OrdTop_{\OC}^{\up}$, $\OrdTop_{\OC}^{\down}$ and $\OrdTop_{\OC}$ for the full subcategories of $\OrdTop$ of ordered spaces with open upper cones, open lower cones, and open cones, respectively.
\end{definition}

We work exclusively with $\OrdTop_{\OC}$ until \cref{section:upper-lower adjunctions}, where all proofs are separated into upper and lower cases. Our category $\OrdTop_{\OC}^{\down}$ is equal to the category $\cat{OOrdTop}$ from \cite[Section~1]{tholen2009OrderedTopologicalStructures}. We will see in \cref{sec:opens} that the open cone condition provides exactly the right connection between order and topology to obtain a well-defined adjunction.

Any preordered set, regarded as a discrete topological space, has open cones.
On the other extreme, on a codiscrete topological space $S$, any preorder has open cones, because $\down S=\down S = S$ and $\up \varnothing = \down \varnothing =\varnothing$ are always open. The next examples are less trivial, and show that the open cone condition is different from the more well known order-separation axioms.

\begin{example}\label{example:open cones with equality}
  Any topological space becomes an ordered space with equality as the preorder.
  This ordered space has open cones because $\down U = U$ and $\up U = U$.
  Thus there are examples of topological spaces with poorly behaved separation properties that nevertheless have open cones. This also tells us that the open cone condition does not imply any of the order-separation axioms~\cite{mccartan1968SeparationAxiomsTopological}.
\end{example}

\begin{example}\label{example:no open cones}
	Not every ordered space has open cones. Consider the real line $\mathbb{R}$ with the standard Euclidean topology. Imagine there is a discrete point living ``before'' the origin of this line. Formally, we get the space $\lbrace\ast\rbrace\sqcup \mathbb{R}$, with the preorder generated by $\ast\leq 0$. Now $\up \{\ast\} = \{\ast\} \sqcup \{0\}$, but of course the singleton $\{0\}$ is not open in $\mathbb{R}$, so this cone is not open. See also \Cref{fig:Lambda}. (For another example see~\cite[Example 10.1.1]{richmond2020GeneralTopologyIntroduction}.)
	
	On the other hand, it is easy to see that the graph of this preorder is closed. Namely, it is the union of the diagonals of $\mathbb{R}$ and $\lbrace \ast\rbrace$ (which are both closed since they are both Hausdorff), together with the singleton $\lbrace (\ast,0)\rbrace$, which is also closed. Hence this is an example of a \emph{T$_2$-ordered} space \cite[Theorem 2]{mccartan1968SeparationAxiomsTopological}, also known as a \emph{pospace}. This, together with the previous example, shows that our condition of having open cones really is different than imposing (order-)separation axioms.
\end{example}

\begin{example}\label{example:specialisation order}
  Any preorder induces the \emph{upper topology} on its carrier set, where a subset $U$ is open if and only if it is upward closed: if $x \in U$ and $x \leq y$, then also $y \in U$. In this way any preordered set becomes an ordered space. By construction it has open upper cones, but its lower cones are open if and only if $\down \up A = \up \down \up A$ for any subset $A$. Any monotone function is automatically continuous for the upper topology.

  Any lattice, that is, any partially ordered set with binary least upper bounds and binary greatest lower bounds, induces the \emph{interval topology} on its carrier set, where a subset is open if and only if it is a union of intervals $\langle x,z \rangle = \{ y \mid x \leq y \leq z\}$. The resulting ordered space has open cones, because
  \[
    \up \bigcup \langle x_i,z_i \rangle = \bigcup \up \langle x_i,z_i \rangle = \bigcup \{ \langle x_i,z \rangle \mid x_i \leq z \},
  \]
  and similarly for past cones.

  Any topological space has a \emph{specialisation} preorder, where $x \leq y$ if and only if $x \in \overline{\{y\}}$. In this way any topological space becomes an ordered space, that always has open future cones, but need not have open past cones. 
Any continuous function is automatically monotone for the specialisation preorders.
\end{example}

Before moving on to physically motivated examples, we record an important characterisation of the open cone condition.

\begin{proposition}\label{proposition:open cones iff internal cones}
	An ordered space has open upper/lower cones if and only if the following holds, respectively:
		\begin{itemize}
			\item if $x\leq y$ then $y\in (\up U)^\circ$ for any open neighbourhood $U$ of $x$;
			\item if $x\leq y$ then $x\in (\down V)^\circ$ for any open neighbourhood $V$ of $y$.
		\end{itemize}
\end{proposition}
Here $A\interior$ denotes the interior of a subset $A\subseteq S$.
\begin{proof}
	Suppose the ordered space has open upper cones, and let $x\leq y$. We need to show that $y\in (\up U)\interior$ for any open neighbourhood $U$ of $x$. But since $x\leq y$ we have $y\in \up \{x\}$, so if $U\ni x$ is open we get $y\in \up U= (\up U)\interior$, as desired.

	For the converse, let $U$ be an arbitrary open set. We need to show ${\up U \subseteq (\up U)^\circ}$. Take an element $y\in \up U$, meaning that there exists $x\in U$ with $x\leq y$. Now the assumption gives precisely that $y\in (\up U)\interior$. Hence $\up U=(\up U)\interior$.
	
	The proof for the lower cones is entirely analogous.
\end{proof}

\begin{figure}[b]
  \centering\begin{tikzpicture}[yscale=.7]
		\draw[yellow_opencone] (-1,1) to (-1,-1) to (3,-1) to (3,1) to (-1,1);
		\draw[thin_dotted] (-1,1) to (-1,-1);
		\draw[thin_dotted] (3,1) to (3,-1);
		\draw (-4,1) to (4,1);
		\node[circle, inner sep=1pt, draw=black, fill=black] at (0,0) {};
		\node[right] at (0,0) {$U = \{(0,0)\}$};
		\draw[dashed] (0,0) to (0,2);
		\node[right] at (0,2) {$\up U$};
		\node[above] at (2.5,-1) {$\down V$};
		\node at (-1,1) {$($};
		\node at (3,1) {$)$};
		\draw[thick] (-1,1) to (3,1);
		\node[above] at (2.5,1) {$V$};
  \end{tikzpicture}
  \caption{An ordered space that satisfies~\eqref{eq:Lambda} but does not have open cones: $U = \{(0,0)\}$ is open, but $\up U = \{(0,0),(0,1)\}$ is not.}
  \label{fig:Lambda}
\end{figure}
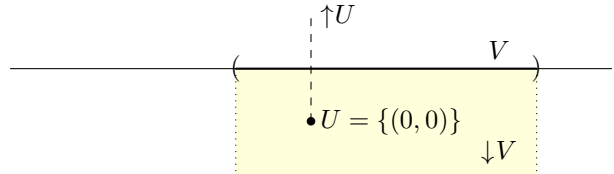

\begin{example}\label{example:failure of Lambda}
		A consequence of \cref{proposition:open cones iff internal cones} is that $S$ has open cones if and only if
		\[
  		U\cap \up\lbrace x\rbrace \subseteq \up\left((\down U)\interior\cap\lbrace x\rbrace\right)
  		\qquad\text{and}\qquad
  		U\cap \down\lbrace x\rbrace \subseteq \down\left((\up U)\interior\cap\lbrace x\rbrace\right)
		\]
		for all $x\in S$ and $U\in\Opens S$. The natural localic generalisation of this condition is to say that
		\begin{equation}\label{eq:Lambda}\tag{$\ast$}
  		U\cap \up V \subseteq \up\left((\down U)\interior\cap V\right)
	  	\qquad\text{and}\qquad
		  U\cap \down V \subseteq \down\left((\up U)\interior\cap V\right)
		\end{equation}
		holds for all $U,V\in\Opens S$. The following is an example of an ordered space that satisfies~\eqref{eq:Lambda}, but does \emph{not} have open cones.
		
		Recall the ordered space in \cref{example:no open cones}, which we now exhibit as the subspace 
		\[
		S= \left(\lbrace 0\rbrace \times \lbrace 0 \rbrace\right)\cup \left(\mathbb{R}\times\lbrace  1\rbrace\right)\subseteq \mathbb{R}^2,
		\]
		where the order $\leq$ on $\mathbb{R}^2$ is given by $(x,y)\leq (a,b)$ if and only if $x=a$ and ${y\leq b}$. Take the topology on $\mathbb{R}^2$ generated by the basis consisting of subsets of the form $(a,b)\times (-\infty,y)\subseteq \mathbb{R}^2$, for arbitrary $a,b,y\in\mathbb{R}$. With this new topology, $\mathbb{R}^2$ has open cones, but $S$ does not. However, after some elementary checks, it can be seen that $S$ does satisfy~\eqref{eq:Lambda}. See also \Cref{fig:Lambda}.
\end{example}

We end this section by considering examples from relativity theory, which were our original motivation. In this setting, the upset $\up \{x\}$ of a point $x$ is typically called its \emph{causal future} and denoted $J^+(x)$, and $\down \{x\}$ is called the \emph{causal past} and denoted $J^-(x)$. Similarly, the interior $(\up \{x\})\interior$ is called the \emph{chronological future} and written $I^+(x)$, and $(\down \{x\})\interior$ is called the \emph{chronological past} and written $I^-(x)$. Extending this notation with $J^\pm(U)=\bigcup_{x \in U} J^\pm(x)$ and $I^\pm(U)=J^\pm(U)\interior$, the open cone condition says $I^\pm(U)=J^\pm(U)$ for all opens $U$. In the setting of relativity theory, $I^\pm$ are not intrinsically defined as the interior of $J^\pm$, but rather this is a derived property~\cite[Theorem 2.27]{minguzzi2019LorentzianCausalityTheory}, that we here lift to a definition. 

We will illustrate our results with diagrams taken from relativity theory intuition: one time dimension runs upwards, and one space dimension is drawn horizontally. Future and past cones of points and open sets then look as in \Cref{figure:illustrations}.

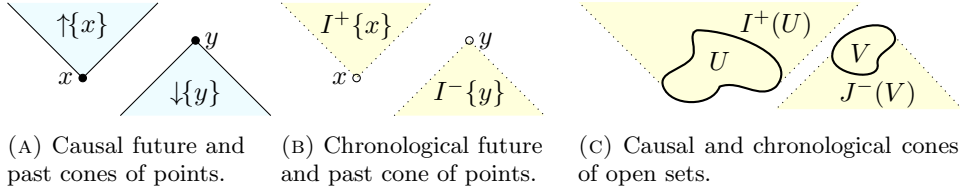
\begin{figure}[t]
	\centering
	\begin{subfigure}[b]{0.25\textwidth}
		\begin{tikzpicture}
			\draw[blue_cone] (0,0) to (1,1) to (-1,1) to (0,0);	
			\draw (-1,1) to (0,0) to (1,1);
			\node[circle, inner sep=1pt, draw=black, fill=black] at (0,0) {};
			\node[left] at (0,0) {$x$};
			\node at (0,.7) {$\up \{x\}$};
			\draw[blue_cone] (.5,-.5) to (1.5,.5) to (2.5,-.5) to (.5,-.5);
			\draw (.5,-.5) to (1.5,.5) to (2.5,-.5);
			\node[circle, inner sep=1pt, draw=black, fill=black] at (1.5,.5) {};
			\node[right] at (1.5,.5) {$y$};
			\node at (1.5,-.2) {$\down \{y\}$};
		\end{tikzpicture}
		\caption{Causal future and past cones of points.}
	\end{subfigure}\hfill
	\begin{subfigure}[b]{0.275\textwidth}
		\begin{tikzpicture}
			\draw[yellow_opencone] (0,0) to (1,1) to (-1,1) to (0,0);	
			\draw[thin_dotted] (-1,1) to (0,0) to (1,1);
			\node[circle, inner sep=1pt, draw=black, fill=none] at (0,0) {};
			\node[left] at (0,0) {$x$};
			\node at (0,.7) {$I^+ \{x\}$};
			\draw[yellow_opencone] (.5,-.5) to (1.5,.5) to (2.5,-.5) to (.5,-.5);
			\draw[thin_dotted] (.5,-.5) to (1.5,.5) to (2.5,-.5);
			\node[circle, inner sep=1pt, draw=black, fill=none] at (1.5,.5) {};
			\node[right] at (1.5,.5) {$y$};
			\node at (1.5,-.2) {$I^- \{y\}$};
		\end{tikzpicture}
		\caption{Chronological future and past cone of points.}
	\end{subfigure}\hfill
	\begin{subfigure}[b]{0.4\textwidth}
		\begin{tikzpicture}[scale=.75]
			\draw[yellow_opencone] (-1.5,1.5) to (0,0) to (1,.5) to (2,0) to (3.5,1.5);
			\draw[yellow_open, thick] (0,0) to[out=110,in=-130] (.5,.5) to[out=50,in=200] (.6,1) to[out=0,in=50] (2,0) to[out=-140,in=20] (1,0) to[out=-110,in=-70] (0,0);			
			\draw[thin_dotted] (0,0) to[out=110,in=-130] (.5,.5) to[out=50,in=180] (.6,1) to[out=0,in=50] (2,0) to[out=-140,in=20] (1,0) to[out=-110,in=-70] (0,0);			
			\draw[thin_dotted] (0,0) to (-1.5,1.5);
			\draw[thin_dotted] (2,0) to (3.5,1.5);
			\node at (1,.45) {$U$};
			\node at (2,1.1) {$I^+(U)$};
			\draw[yellow_opencone] (5.3,-.4) to (4,1) to (3,.6) to (2,-.4);
			\draw[thin_dotted] (3,.6) to (2,-.4);
			\draw[thin_dotted] (4,1) to (5.3,-.4);
			\draw[yellow_open, thick] (3,.5) to[out=100,in=110] (4.1,.9) to[out=-70,in=100] (3.8,.4) to[out=-80,in=-80] (3,.5);
			\draw[thin_dotted] (3,.5) to[out=100,in=110] (4.1,.9) to[out=-70,in=100] (3.8,.4) to[out=-80,in=-80] (3,.5);
			\node at (3.5,.6) {$V$};
			\node at (3.8,-.12) {$J^-(V)$};
		\end{tikzpicture}
		\caption{Causal and chronological cones of open sets.}
	\end{subfigure}
\caption{Illustrations of $\leq$ in a Minkowski-like ordered space.}
\label{figure:illustrations}
\end{figure}

The following lemma shows that the open cone condition can be regarded as abstracting the \emph{``push-up'' principle} from Lorentzian causality theory~\cite{grant2020FutureNotAlwaysOpen,kronheimer1967structure}.

\begin{lemma}\label{lemma:open cones iff JI=I}
  An ordered space has open upper cones if and only if ${\up (\up U)\interior \subseteq (\up U)\interior}$ for all opens $U$, and open lower cones if and only if
  $\down (\down U)\interior \subseteq (\down U)\interior$ for all opens~$U$.
\end{lemma}
\begin{proof}
  Suppose that the stated condition for upper cones holds, and let $U$ be an open subset.
  By construction $(\up U)\interior \subseteq \up U$, so we need to show the converse inclusion.
  From \Cref{lem:cones}(a) we know that $U \subseteq \up U$, so openness of $U$ gives $U=U\interior \subseteq (\up U)\interior$. 
  Because upsets are monotone with respect to subset inclusion by \cref{lem:cones}(c), this gives $\up U \subseteq \up (\up U)\interior \subseteq (\up U) \interior$, as desired.

	Conversely, if we have open upper cones, then
	\[
    \up ( \up U)\interior = \up \up U = \up U = (\up U)\interior
	\]
	follows from \cref{lem:cones}(b). The proof for lower cones is analogous.
\end{proof}

A \emph{smooth spacetime} is a four-dimensional connected smooth Lorentzian manifold that is time oriented~\cite[Definition 5.3]{landsman2021GR}. We may define the \emph{causality relation} $x \leq y$ on points of a spacetime when there exists a \emph{future directed causal curve} from $x$ to $y$: that is, a smooth curve whose tangent vector is future directed and causal at every point (see \emph{e.g.} \cite[Section 5.3]{landsman2021GR} or \cite{minguzzi2019LorentzianCausalityTheory} for more details).

\begin{corollary}\label{corollary:smooth spacetimes have OC}
	Any smooth spacetime has open cones.
\end{corollary}		
\begin{proof}
	Combine \cref{lemma:open cones iff JI=I} with \cite[Proposition 5.4]{landsman2021GR}.
\end{proof}

We expect that the previous corollary continues to hold for a large class of lower-regularity spacetimes; essentially those where the push-up principle continues to hold. However, extremely low-regularity spacetimes potentially exhibit more degenerate causal structure, see \textit{e.g.}~\cite{chrusciel2012LorentzianCausalityContinuous,grant2020FutureNotAlwaysOpen} and references therein.
The previous corollary shows that $\OrdTop_\OC$ may serve as a viable category of abstract (smooth) spacetimes, and also provides many interesting examples.

\section{Ordered locales}\label{section:ordered locales}

This section discusses the notion of \emph{ordered locales}. These axiomatise how the structure of an ordered space transfers to its lattice of open subsets, as will be discussed in the next section. The main idea is to introduce a new order $\Leq$ between open subsets, where $U\Leq V$ if and only if the elements of $U$ and $V$ are suitably related. There are essentially three canonical choices for this~\cite[Section~11.1]{vickers1996topology}:
	\begin{itemize}
		\item the \emph{upper order:} $U\LeqU V$ if and only if $V\subseteq \up U$;
		\item the \emph{lower order:} $U\LeqL V$ if and only if $U\subseteq \down V$;
		\item the \emph{Egli-Milner order}: $U\LeqEM V$ if and only if $U\LeqU V$ and $U\LeqL V$.
	\end{itemize}
Hence the interpretation of $U\LeqEM V$ is that every element of $U$ lies in the past of some element of $V$, and every element of $V$ lies in the future of some element of $U$. Since we assume $U$ and $V$ to be open, this can equivalently be expressed as: $U\subseteq (\down V)\interior$ and $V\subseteq (\up U)\interior$. See \cref{figure:illustration of causal order on opens} for the visual intuition. The behaviour of $\leq$ transfers to properties of these orders, which we will capture as the axioms for ordered locales below.
\begin{figure}
	\centering
	\begin{subfigure}[b]{0.29\textwidth}
		{\tikzstyle{every picture}=[tikzfig]\begin{tikzpicture}
	\begin{pgfonlayer}{nodelayer}
		\node [style=none] (31) at (-0.5, 3.25) {};
		\node [style=none] (32) at (0.65, 3.5) {};
		\node [style=none] (33) at (0.9, 2.4) {};
		\node [style=none] (34) at (0, 2) {};
		\node [style=none] (27) at (-1.25, 3.25) {};
		\node [style=none] (28) at (1.75, 2) {};
		\node [style=none] (29) at (3, -2.5) {};
		\node [style=none] (30) at (-3, -2.5) {};
		\node [style=none] (23) at (-3, 4.25) {};
		\node [style=none] (24) at (3, 4.25) {};
		\node [style=none] (25) at (1.75, -1.5) {};
		\node [style=none] (26) at (-1.75, -1.75) {};
		\node [style=none] (0) at (-1, -0.75) {};
		\node [style=none] (1) at (0.25, 0) {};
		\node [style=none] (2) at (1.5, -1.75) {};
		\node [style=none] (3) at (-0.25, -2) {};
		\node [style=none] (4) at (-1.75, -1.75) {};
		\node [style=none] (6) at (0, -1) {$U$};
		\node [style=none] (7) at (-1, 1.75) {};
		\node [style=none] (8) at (0.75, 1.5) {};
		\node [style=none] (9) at (1.75, 2) {};
		\node [style=none] (10) at (0.5, 2.75) {};
		\node [style=none] (11) at (0, 4) {};
		\node [style=none] (12) at (-1.25, 3.25) {};
		\node [style=none] (13) at (-0.25, 2.75) {$V$};
		\node [style=none] (14) at (-3, -2.5) {};
		\node [style=none] (15) at (3, -2.5) {};
		\node [style=none] (16) at (0.65, 3.5) {};
		\node [style=none] (17) at (0.9, 2.4) {};
		\node [style=none] (18) at (-3, 4.25) {};
		\node [style=none] (19) at (3, 4.25) {};
		\node [style=none] (20) at (1.75, -1.5) {};
		\node [style=none] (21) at (-3.75, 2.5) {$(\up U)\interior$};
		\node [style=none] (22) at (4, -1.25) {$(\down V)\interior$};
	\end{pgfonlayer}
	\begin{pgfonlayer}{edgelayer}
		\draw [style={blue_cone}] (34.center)
			 to (31.center)
			 to (32.center)
			 to (33.center)
			 to cycle;
		\draw [style={blue_cone}] (29.center)
			 to (30.center)
			 to (27.center)
			 to (28.center)
			 to cycle;
		\draw [style={yellow_cone_light}] (24.center)
			 to (25.center)
			 to (26.center)
			 to (23.center)
			 to cycle;
		\draw [style={yellow_open}] (1.center)
			 to [in=30, out=-15] (2.center)
			 to [in=45, out=-150] (3.center)
			 to [in=-90, out=-135] (4.center)
			 to [in=240, out=90] (0.center)
			 to [in=165, out=75] cycle;
		\draw [style={blue_open}] (8.center)
			 to [in=-90, out=15] (9.center)
			 to [in=-120, out=90] (10.center)
			 to [in=345, out=60, looseness=1.25] (11.center)
			 to [bend left=315] (12.center)
			 to [bend right] (7.center)
			 to [in=-165, out=-45, looseness=0.75] cycle;
		\draw [style={thin_dotted}] (12.center) to (14.center);
		\draw [style={thin_dotted}] (9.center) to (15.center);
		\draw [style={thin_dotted}] (16.center) to (17.center);
		\draw [style={thin_dotted}] (4.center) to (18.center);
		\draw [style={thin_dotted}] (20.center) to (19.center);
	\end{pgfonlayer}
\end{tikzpicture}}
		\caption{Archetypal $U\LeqEM V$.}
	\end{subfigure}\hfill
	\begin{subfigure}[b]{0.33\textwidth}
		{\tikzstyle{every picture}=[tikzfig]\begin{tikzpicture}
	\begin{pgfonlayer}{nodelayer}
		\node [style=none] (38) at (-1.15, 2.5) {};
		\node [style=none] (39) at (3.6, 2.5) {};
		\node [style=none] (40) at (4.5, -2.5) {};
		\node [style=none] (41) at (-2, -2.5) {};
		\node [style=none] (16) at (-2.5, 4.25) {};
		\node [style=none] (17) at (2.3, 4.25) {};
		\node [style=none] (18) at (-1, -1.45) {};
		\node [style=none] (19) at (0.75, -1.5) {};
		\node [style=none] (0) at (-1, -1) {};
		\node [style=none] (1) at (-0.25, -1.75) {};
		\node [style=none] (2) at (0.75, -1.5) {};
		\node [style=none] (3) at (0.5, -0.75) {};
		\node [style=none] (4) at (-0.25, -0.25) {};
		\node [style=none] (5) at (-0.23, -1) {$U$};
		\node [style=none] (12) at (-0.5, 2.25) {$V$};
		\node [style=none] (13) at (-1, -1.45) {};
		\node [style=none] (14) at (-2.5, 4.25) {};
		\node [style=none] (15) at (2.3, 4.25) {};
		\node [style=none] (20) at (-2, -2.5) {};
		\node [style=none] (21) at (4.5, -2.5) {};
		\node [style=none] (25) at (1.71, 2) {};
		\node [style=none] (26) at (2.75, 1.75) {};
		\node [style=none] (27) at (2.75, 2.75) {};
		\node [style=none] (28) at (1.84, 2.5) {};
		\node [style=none] (29) at (3.6, 2.5) {};
		\node [style=none] (30) at (-1.15, 2.5) {};
		\node [style=none] (31) at (-0.65, 1.75) {};
		\node [style=none] (32) at (1.71, 2) {};
		\node [style=none] (35) at (1.84, 2.5) {};
		\node [style=none] (36) at (-3.25, 2.5) {$(\up U)\interior$};
		\node [style=none] (37) at (2.75, -1.25) {$(\down V)\interior$};
	\end{pgfonlayer}
	\begin{pgfonlayer}{edgelayer}
		\draw [style={blue_cone}] (39.center)
			 to (40.center)
			 to (41.center)
			 to (38.center)
			 to cycle;
		\draw [style={yellow_cone_light}] (18.center)
			 to (16.center)
			 to (17.center)
			 to (19.center)
			 to cycle;
		\draw [style={yellow_open}] (1.center)
			 to [in=-135, out=225] (0.center)
			 to [in=165, out=45] (4.center)
			 to [in=90, out=-15, looseness=0.75] (3.center)
			 to [in=75, out=-90, looseness=0.75] (2.center)
			 to [in=45, out=-105, looseness=1.25] cycle;
		\draw [style={thin_dotted}] (13.center) to (14.center);
		\draw [style={thin_dotted}] (2.center) to (15.center);
		\draw [style={red_open}] (28.center)
			 to (25.center)
			 to [in=180, out=0] (26.center)
			 to [in=15, out=0, looseness=3.00] (27.center)
			 to [in=0, out=-165] cycle;
		\draw [style={blue_open}] (35.center)
			 to [in=75, out=-180, looseness=0.75] (30.center)
			 to [in=180, out=-105, looseness=1.25] (31.center)
			 to [in=180, out=0] (32.center)
			 to cycle;
		\draw [style={thin_dotted}] (30.center) to (20.center);
		\draw [style={thin_dotted}] (29.center) to (21.center);
	\end{pgfonlayer}
\end{tikzpicture}}
		\caption{$U\LeqL V$, but $U\not\LeqU V$.}
	\end{subfigure}\hfill
	\begin{subfigure}[b]{0.29\textwidth}
		{\tikzstyle{every picture}=[tikzfig]\begin{tikzpicture}
	\begin{pgfonlayer}{nodelayer}
		\node [style=none] (68) at (-3.25, 1.25) {};
		\node [style=none] (69) at (0.75, 0.75) {};
		\node [style=none] (70) at (1.75, -2.5) {};
		\node [style=none] (71) at (-4.25, -2.5) {};
		\node [style=none] (57) at (-1.25, 0.2) {$U$};
		\node [style=none] (58) at (-3.75, 4.25) {};
		\node [style=none] (59) at (1.25, 4.25) {};
		\node [style=none] (60) at (0.15, 0.5) {};
		\node [style=none] (61) at (-3.75, 4.25) {};
		\node [style=none] (62) at (1.25, 4.25) {};
		\node [style=none] (63) at (0.15, 0.5) {};
		\node [style=none] (64) at (-2.5, 0.25) {};
		\node [style=none] (65) at (-4.5, 2.75) {$(\up U)\interior$};
		\node [style=none] (66) at (-4.25, -2.5) {};
		\node [style=none] (67) at (1.75, -2.5) {};
		\node [style=none] (72) at (0.25, -1.5) {$(\down V)\interior$};
		\node [style=none] (44) at (-3.25, 1.25) {};
		\node [style=none] (45) at (-3.25, 0) {};
		\node [style=none] (46) at (-1.25, -0.75) {};
		\node [style=none] (47) at (0.25, -0.25) {};
		\node [style=none] (48) at (0.75, 0.75) {};
		\node [style=none] (49) at (0, 2.25) {};
		\node [style=none] (50) at (-2.25, 2.75) {};
		\node [style=none] (51) at (-2.5, 0.75) {};
		\node [style=none] (52) at (-1.25, 0.75) {};
		\node [style=none] (53) at (0, 0.75) {};
		\node [style=none] (54) at (0, 0.25) {};
		\node [style=none] (55) at (-1.25, -0.25) {};
		\node [style=none] (56) at (-2.5, 0.25) {};
		\node [style=none] (77) at (-3.25, 1.25) {};
		\node [style=none] (78) at (-3.25, 0) {};
		\node [style=none] (79) at (-1.25, -0.75) {};
		\node [style=none] (80) at (0.25, -0.25) {};
		\node [style=none] (81) at (0.75, 0.75) {};
		\node [style=none] (90) at (0.15, 0.5) {};
		\node [style=none] (95) at (0, 0.25) {};
		\node [style=none] (96) at (-1.25, -0.25) {};
		\node [style=none] (97) at (-2.51, 0.25) {};
		\node [style=none] (102) at (-3.01, 1.88) {};
		\node [style=none] (103) at (0.48, 1.51) {};
		\node [style=none] (104) at (-1.25, 1.75) {$V$};
	\end{pgfonlayer}
	\begin{pgfonlayer}{edgelayer}
		\draw [style={blue_cone}] (68.center)
			 to (69.center)
			 to (70.center)
			 to (71.center)
			 to cycle;
		\draw [style={yellow_cone_light}] (64.center)
			 to (61.center)
			 to (62.center)
			 to (63.center)
			 to cycle;
		\draw [style={thin_dotted}] (44.center) to (66.center);
		\draw [style={thin_dotted}] (48.center) to (67.center);
		\draw [style={blue_open}] (50.center)
			 to [bend left, looseness=0.75] (49.center)
			 to [in=105, out=-45] (48.center)
			 to [in=30, out=-75] (47.center)
			 to [in=0, out=-150, looseness=0.75] (46.center)
			 to [in=-75, out=180, looseness=0.75] (45.center)
			 to [in=255, out=105, looseness=0.75] (44.center)
			 to [in=-165, out=75, looseness=0.75] cycle;
		\draw [style={yellow_open}] (54.center)
			 to [in=0, out=-150] (55.center)
			 to [in=-60, out=180, looseness=0.50] (56.center)
			 to [bend left=45, looseness=0.75] (51.center)
			 to [in=-180, out=30, looseness=0.75] (52.center)
			 to [in=150, out=0, looseness=0.75] (53.center)
			 to [in=30, out=-30] cycle;
		\draw [style={thin_dotted}] (56.center) to (58.center);
		\draw [style={thin_dotted}] (60.center) to (59.center);
		\draw [style={red_open}] (102.center)
			 to [in=105, out=-75, looseness=0.25] (97.center)
			 to [in=180, out=-60, looseness=0.50] (96.center)
			 to [in=-150, out=0] (95.center)
			 to (90.center)
			 to (103.center)
			 to (81.center)
			 to [in=30, out=-75] (80.center)
			 to [in=0, out=-150, looseness=0.75] (79.center)
			 to [in=-75, out=180, looseness=0.75] (78.center)
			 to [in=255, out=105, looseness=0.75] (77.center)
			 to [in=-120, out=75, looseness=0.75] cycle;
	\end{pgfonlayer}
\end{tikzpicture}}
		\caption{$U\subseteq V$, but $U\not\LeqU V$.}
	\end{subfigure}
\caption{Illustrations of $\Leq$ in a Minkowski-like ordered space.}
\label{figure:illustration of causal order on opens}
\end{figure}
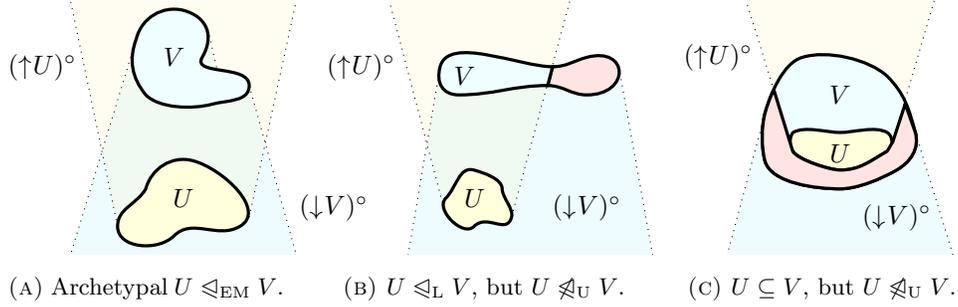

Let us first briefly recall basic properties of the category of locales; see~\cite{picado2012frames} for more details.
A \emph{frame} is a complete lattice $L$ satisfying \emph{infinite distributivity}:
\[
	U\wedge \bigvee V_i = \bigvee (U\wedge V_i).
\]
A \emph{frame morphism} is a function that preserves all suprema and finite meets, so in particular preserves the least and greatest elements. Write $\Frame$ for the category of frames and frame morphisms. The category of \emph{locales} is the opposite:
\[
	\Loc = \Frame\op.
\]
Thus a locale $X$ is formally defined by its \emph{frame of opens} $\Opens X$, and a locale morphism $f \colon X\to Y$ corresponds to a frame morphism $f^{-1}\colon \Opens Y\to \Opens X$, interpreted as the `preimage' map of $f$. 
We will write $\sqleq$ for the intrinsic order of a frame $\Opens X$ (to prevent a notation clash with the order $\leq$ on a topological space).

Now we come to the main definition.

\begin{definition}\label{definition:ordered locales}
	An \emph{ordered locale} is a locale $X$ together with a preorder $\Leq$ on its frame $\Opens X$ of opens, satisfying 
	\begin{equation}
	  \forall i\in I \colon\; U_i \Leq V_i
	  \;\implies\;
	  \bigvee_{i \in I} U_i \Leq \bigvee_{i \in I} V_i,
	  \tag{$\vee$}\label{axiom:V}
	\end{equation}
	where $I$ is an arbitrary indexing set.
\end{definition}

\begin{example}\label{example:locale with inclusion}
	The frame of opens $\Opens X$ of any locale $X$ comes equipped with an intrinsic order: the inclusion relation $\sqleq$. From the infinite distributive law it follows that axiom~\eqref{axiom:V} is satisfied, so $(X,\sqleq)$ is an ordered locale. This example turns out to be related to the specialisation order (mentioned in \cref{example:specialisation order}), and we revisit it in \cref{example:ordered locale with inclusion}.
\end{example}

\begin{lemma}\label{lem:axiomL}
  In an ordered locale:
  \begin{enumerate}
	  \item[(a)] if $U \Leq U'$ and $U \sqleq V$, then $V \Leq V'$ and $U' \sqleq V'$ for some $V'$;
	  \item[(b)] if $U \Leq U'$ and $U' \sqleq V'$, then $U \sqleq V$ and $V \Leq V'$ for some $V$.
  \end{enumerate}
  Diagrammatically:
  \[
		\begin{tikzcd}[every label/.append style = {font = \normalsize},column sep=0.25cm, row sep=0.2cm]
			{U} & {U'} \\
			{V} & {\exists V'}
			\arrow["\sqleq"{anchor=center, rotate=-90}, draw=none, from=2-1, to=1-1]
			\arrow["\sqleq"{anchor=center, rotate=-90}, draw=none, from=2-2, to=1-2]
			\arrow["{}"{description}, "\Leq"{anchor=center}, draw=none, from=1-1, to=1-2]
			\arrow["{}"{description}, "\Leq"{anchor=center}, draw=none, from=2-1, to=2-2]
		\end{tikzcd}  	
		\qquad \qquad
		\begin{tikzcd}[every label/.append style = {font = \normalsize},column sep=0.25cm, row sep=0.2cm]
			{U} & {U'} \\
			{\exists V} & {V'.}
			\arrow["\sqleq"{anchor=center, rotate=-90}, draw=none, from=2-1, to=1-1]
			\arrow["\sqleq"{anchor=center, rotate=-90}, draw=none, from=2-2, to=1-2]
			\arrow["{}"{description}, "\Leq"{anchor=center}, draw=none, from=1-1, to=1-2]
			\arrow["{}"{description}, "\Leq"{anchor=center}, draw=none, from=2-1, to=2-2]
		\end{tikzcd}
  \]
\end{lemma}
See also \cref{figure:illustration of axioms}.
\begin{proof}
  For (a), since $U\Leq U'$ and $V\Leq V$, axiom~\eqref{axiom:V} implies $V=U\vee V \Leq U'\vee V$, and so $V':= U'\vee V$ satisfies the statement. Point (b) is proved similarly.
\end{proof}

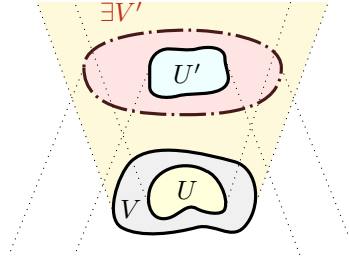
\begin{figure}
	\centering
		{\tikzstyle{every picture}=[tikzfig]\begin{tikzpicture}
	\begin{pgfonlayer}{nodelayer}
		\node [style=none] (52) at (-2, -1.75) {};
		\node [style=none] (53) at (1.775, -1.75) {};
		\node [style=none] (54) at (-4, 3.75) {};
		\node [style=none] (55) at (3.75, 3.75) {};
		\node [style=none] (40) at (0, 0.75) {};
		\node [style=none] (41) at (-2.75, 2) {};
		\node [style=none] (42) at (0, 3) {};
		\node [style=none] (43) at (2.5, 2) {};
		\node [style=none] (44) at (-4.75, -3) {};
		\node [style=none] (45) at (4.5, -3) {};
		\node [style=none] (30) at (-1, 1.5) {};
		\node [style=none] (31) at (-0.25, 1.25) {};
		\node [style=none] (32) at (-1, 2) {};
		\node [style=none] (33) at (-0.5, 2.5) {};
		\node [style=none] (35) at (1, 1.5) {};
		\node [style=none] (36) at (1, 2.25) {};
		\node [style=none] (37) at (0.25, 2.5) {};
		\node [style=none] (38) at (-3, -3) {};
		\node [style=none] (39) at (3, -3) {};
		\node [style=none] (47) at (0, -1.25) {$U$};
		\node [style=none] (48) at (0, 1.85) {$U'$};
		\node [style=none] (49) at (-1.5, -1.75) {$V$};
		\node [style={RRed_text}] (51) at (-1.75, 3.5) {$\exists V'$};
		\node [style=none] (61) at (-1.5, -0.5) {};
		\node [style=none] (62) at (-1.75, -2.25) {};
		\node [style=none] (63) at (0.25, -2.25) {};
		\node [style=none] (64) at (1.5, -2.25) {};
		\node [style=none] (65) at (1.25, -0.5) {};
		\node [style=none] (56) at (-1, -1.25) {};
		\node [style=none] (57) at (-0.5, -2) {};
		\node [style=none] (58) at (0.25, -1.75) {};
		\node [style=none] (59) at (1, -1.75) {};
		\node [style=none] (60) at (0.5, -0.75) {};
		\node [style=none] (66) at (-1, -1.75) {};
		\node [style=none] (67) at (-3, 3.75) {};
		\node [style=none] (68) at (1, -1.75) {};
		\node [style=none] (69) at (2.75, 3.75) {};
	\end{pgfonlayer}
	\begin{pgfonlayer}{edgelayer}
		\draw [style={yellow_cone}] (52.center)
			 to (54.center)
			 to (55.center)
			 to (53.center)
			 to cycle;
		\draw [style={red_dashed}] (41.center)
			 to [in=180, out=-90] (40.center)
			 to [in=-90, out=0] (43.center)
			 to [in=0, out=90, looseness=0.75] (42.center)
			 to [in=90, out=180, looseness=0.75] cycle;
		\draw [style={blue_open}] (30.center)
			 to [in=-165, out=-90] (31.center)
			 to [in=-135, out=15, looseness=0.75] (35.center)
			 to [in=-90, out=45] (36.center)
			 to [in=375, out=90] (37.center)
			 to [in=0, out=195, looseness=0.75] (33.center)
			 to [in=90, out=-180] (32.center)
			 to cycle;
		\draw [style={thin_dotted}] (41.center) to (44.center);
		\draw [style={thin_dotted}] (43.center) to (45.center);
		\draw [style={thin_dotted}] (36.center) to (39.center);
		\draw [style={thin_dotted}] (32.center) to (38.center);
		\draw [style={grey_open}] (65.center)
			 to [in=45, out=-15] (64.center)
			 to [in=15, out=-135] (63.center)
			 to [in=330, out=-165, looseness=0.75] (62.center)
			 to [in=-135, out=150, looseness=0.75] (61.center)
			 to [in=165, out=45, looseness=0.75] cycle;
		\draw [style={yellow_open}] (58.center)
			 to [in=15, out=150] (57.center)
			 to [in=-105, out=-165, looseness=1.25] (56.center)
			 to [in=150, out=75] (60.center)
			 to [bend left=45, looseness=0.75] (59.center)
			 to [in=-30, out=-105, looseness=0.75] cycle;
		\draw [style={thin_dotted}] (52.center) to (54.center);
		\draw [style={thin_dotted}] (53.center) to (55.center);
		\draw [style={thin_dotted}] (66.center) to (67.center);
		\draw [style={thin_dotted}] (68.center) to (69.center);
	\end{pgfonlayer}
\end{tikzpicture}}
	\caption{Illustration of \Cref{lem:axiomL}(a).}
	\label{figure:illustration of axioms}
\end{figure}

Morphisms of ordered locales are more intricate to define than in the setting of topological spaces, because a locale morphism $f \colon X\to Y$ is a frame morphism $f^{-1} \colon \Opens Y\to \Opens X$ in the other direction. To ask that $f^{-1}$ preserves order makes little sense, as it corresponds to asking that the would-be function $f$ \emph{reflects} order. 
To remedy this we move from functions to \emph{relations}.
For a locale morphism $f \colon X\to Y$, define the relation:
\[
	R_f := \{ (U,V) \mid U\sqleq f^{-1}V\} \subseteq \Opens X \times \Opens Y.
\]
Denoting the category of sets and relations by $\Rel$, we can similarly interpret a preorder $\Leq$ on $\Opens X$ as a morphism $\Opens X \to \Opens X$. In fact, $\Rel$ is a 2-category, with a unique 2-morphism $R\Rightarrow S$ precisely when $R\subseteq S$. Finally, $\Rel$ has a \emph{dagger}, where every relation $R \colon A\to B$ gives rise to another relation $R^\dag \colon B\to A$ defined by $(b,a)\in R^\dag$ if and only if $(a,b) \in R$. 

\begin{definition}\label{definition:morphisms of ordered locales}
	Consider a morphism of locales $f\colon X\to Y$ between ordered locales. We say $f$ is \emph{upper/lower monotone} if the following 2-cells in $\Rel$ exist, respectively:	
		\begin{equation*}
			\begin{tikzcd}
				{\Opens X} & {\Opens X} \\
				{\Opens Y} & {\Opens Y}
				\arrow["{\Leq}", from=1-1, to=1-2]
				\arrow["{\Leq}"', from=2-1, to=2-2]
				\arrow["{R_f^\dagger}", from=2-1, to=1-1]
				\arrow["{R_f^\dagger}"', from=2-2, to=1-2]
				\arrow[dashed, Rightarrow, from=1-1, to=2-2]
			\end{tikzcd}
		\qquad \qquad
		\begin{tikzcd}
			{\Opens X} & {\Opens X} \\
			{\Opens Y} & {\Opens Y.}
			\arrow["{R_f}"', from=1-1, to=2-1]
			\arrow["{R_f}", from=1-2, to=2-2]
			\arrow["{\Leq}", from=1-1, to=1-2]
			\arrow["{\Leq}"', from=2-1, to=2-2]
			\arrow[dashed, Rightarrow, from=1-2, to=2-1]
		\end{tikzcd}
	\end{equation*}
	We say that $f$ is \emph{monotone} if it is both upper and lower monotone.
\end{definition}

The following lemma gives a convenient way to recognise monotonicity of a locale morphism. See also \Cref{figure:illustration of monotone map of locales}.

\begin{lemma}\label{proposition:identities and axioms Lpm}\label{lemma:monotonicity unpacked}
  A map $f$ between ordered locales is upper/lower monotone if and only if, respectively:
  	\begin{itemize}
  		\item if $U \Leq U'$ and $(U,V) \in R_f$, then $(U',V') \in R_f$ and $V \Leq V'$ for some $V'$;
  		\item if $U \Leq U'$ and $(U',V') \in R_f$, then $(U,V) \in R_f$ and $V \Leq V'$ for some $V$.
  	\end{itemize}
  	Diagrammatically:
	  	\[
	  \begin{tikzcd}[every label/.append style = {font = \normalsize},column sep=0.3cm, row sep=0.35cm]
	  	{U} & {U'} \\
	  	{V} & {\exists V'}
	  	\arrow["R_f"{anchor=center}, draw=none, from=2-1, to=1-1]
	  	\arrow["R_f"{anchor=center}, draw=none, from=2-2, to=1-2]
	  	\arrow["{}"{description}, "\Leq"{anchor=center}, draw=none, from=1-1, to=1-2]
	  	\arrow["{}"{description}, "\Leq"{anchor=center}, draw=none, from=2-1, to=2-2]
	  \end{tikzcd}
	  \qquad \qquad 
	  \begin{tikzcd}[every label/.append style = {font = \normalsize},column sep=0.3cm, row sep=0.35cm]
	  	{U} & {U'} \\
	  	{\exists V} & {V'.}
	  	\arrow["R_f"{anchor=center}, draw=none, from=2-1, to=1-1]
	  	\arrow["R_f"{anchor=center}, draw=none, from=2-2, to=1-2]
	  	\arrow["{}"{description}, "\Leq"{anchor=center}, draw=none, from=1-1, to=1-2]
	  	\arrow["{}"{description}, "\Leq"{anchor=center}, draw=none, from=2-1, to=2-2]
	  \end{tikzcd}		
	  \]
\end{lemma}
\begin{proof}
	The upper monotonicity condition for $f$ says that we have an inclusion of graphs: ${\Leq}\circ R_f^\dagger\subseteq R_f^\dagger\circ{\Leq}$. Now $(V,U')$ is in the graph of  ${\Leq}\circ R_f^\dagger$ if and only if there exists $U\in \Opens X$ such that $(U,V)\in R_f$ and $U\Leq U'$. This is exactly the initial data given in the first bullet-point. The inclusion of the graphs then says $(V,U')\in R_f^\dagger\circ {\Leq}$, which holds if and only if there exists $V'\in\Opens Y$ such that $V\Leq V'$ and $(U',V')\in R_f$, precisely as desired. The condition for lower monotonicity is proved similarly.
\end{proof}

\begin{corollary}\label{corollary:identities are monotone}
	Let $X$ be a locale. A preorder $\Leq$ on $\Opens X$ satisfies the property of \cref{lem:axiomL}(a)/(b) if and only if $\id[X]$ is upper/lower monotone.
\end{corollary}

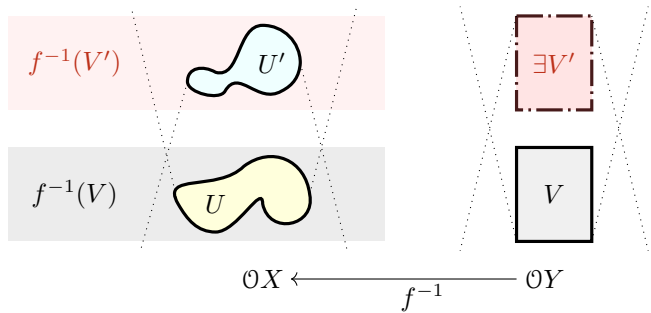
\begin{figure}[b]
	\centering
	{\tikzstyle{every picture}=[tikzfig]\begin{tikzpicture}
	\begin{pgfonlayer}{nodelayer}
		\node [style=none] (59) at (-5.25, 4.75) {};
		\node [style=none] (60) at (-5.25, 2.25) {};
		\node [style=none] (61) at (4.75, 2.25) {};
		\node [style=none] (62) at (4.75, 4.75) {};
		\node [style=none] (55) at (-5.25, 1.25) {};
		\node [style=none] (56) at (-5.25, -1.25) {};
		\node [style=none] (57) at (4.75, -1.25) {};
		\node [style=none] (58) at (4.75, 1.25) {};
		\node [style=none] (19) at (0.25, -0.25) {$U$};
		\node [style=none] (20) at (1.75, 3.5) {$U'$};
		\node [style=none] (28) at (-0.5, 3) {};
		\node [style=none] (29) at (0.5, 2.75) {};
		\node [style=none] (30) at (2, 2.75) {};
		\node [style=none] (31) at (2.5, 3.5) {};
		\node [style=none] (32) at (1.5, 4.5) {};
		\node [style=none] (33) at (0.25, 3.25) {};
		\node [style=none] (34) at (-0.75, -0.25) {};
		\node [style=none] (35) at (0.25, -1) {};
		\node [style=none] (36) at (1.5, -0.25) {};
		\node [style=none] (37) at (2.75, 0) {};
		\node [style=none] (38) at (1, 0.75) {};
		\node [style={RRed_text}] (39) at (-3.5, 3.5) {$f^{-1}(V')$};
		\node [style=none] (40) at (-3.5, 0) {$f^{-1}(V)$};
		\node [style=none] (41) at (-1.75, -1.5) {};
		\node [style=none] (42) at (3.75, -1.5) {};
		\node [style=none] (43) at (-2, 5) {};
		\node [style=none] (44) at (4, 5) {};
		\node [style=none] (45) at (8.25, 1.25) {};
		\node [style=none] (46) at (8.25, -1.25) {};
		\node [style=none] (47) at (10.25, -1.25) {};
		\node [style=none] (48) at (10.25, 1.25) {};
		\node [style=none] (49) at (8.25, 4.75) {};
		\node [style=none] (50) at (10.25, 4.75) {};
		\node [style=none] (51) at (10.25, 2.25) {};
		\node [style=none] (52) at (8.25, 2.25) {};
		\node [style={RRed_text}] (53) at (9.25, 3.5) {$\exists V'$};
		\node [style=none] (54) at (9.25, 0) {$V$};
		\node [style=none] (63) at (6.75, 5) {};
		\node [style=none] (64) at (11.75, 5) {};
		\node [style=none] (65) at (6.75, -1.5) {};
		\node [style=none] (66) at (11.75, -1.5) {};
		\node [style=none] (67) at (1.5, -2.25) {$\Opens X$};
		\node [style=none] (68) at (9, -2.25) {$\Opens Y$};
		\node [style=none] (69) at (2.25, -2.25) {};
		\node [style=none] (70) at (8.25, -2.25) {};
		\node [style=none] (71) at (5.75, -2.75) {$f^{-1}$};
	\end{pgfonlayer}
	\begin{pgfonlayer}{edgelayer}
		\draw [style={red_cone}] (60.center)
			 to (59.center)
			 to (62.center)
			 to (61.center)
			 to cycle;
		\draw [style={grey_cone}] (56.center)
			 to (55.center)
			 to (58.center)
			 to (57.center)
			 to cycle;
		\draw [style={blue_open}] (33.center)
			 to [in=90, out=150, looseness=1.25] (28.center)
			 to [bend right=75, looseness=0.75] (29.center)
			 to [in=-150, out=60, looseness=1.25] (30.center)
			 to [bend right] (31.center)
			 to [bend right=45] (32.center)
			 to [in=-15, out=-180] cycle;
		\draw [style={yellow_open}] (37.center)
			 to [bend left=90, looseness=1.75] (36.center)
			 to [in=0, out=90] (35.center)
			 to [in=300, out=-180] (34.center)
			 to [in=-135, out=120, looseness=1.25] (38.center)
			 to [bend left=60, looseness=1.25] cycle;
		\draw [style={thin_dotted}] (34.center) to (43.center);
		\draw [style={thin_dotted}] (37.center) to (44.center);
		\draw [style={thin_dotted}] (28.center) to (41.center);
		\draw [style={thin_dotted}] (31.center) to (42.center);
		\draw [style={red_dashed}] (49.center)
			 to (50.center)
			 to (51.center)
			 to (52.center)
			 to cycle;
		\draw [style={grey_open}] (46.center)
			 to (45.center)
			 to (48.center)
			 to (47.center)
			 to cycle;
		\draw [style={thin_dotted}] (46.center) to (63.center);
		\draw [style={thin_dotted}] (47.center) to (64.center);
		\draw [style={thin_dotted}] (49.center) to (65.center);
		\draw [style={thin_dotted}] (50.center) to (66.center);
		\draw [style={preimage_frames}] (70.center) to (69.center);
	\end{pgfonlayer}
\end{tikzpicture}}
	\caption{Illustration of monotonicity of locale morphisms.}
	\label{figure:illustration of monotone map of locales}
\end{figure}
	
\begin{proposition}\label{proposition:composition of maps of ordered locales}
  Ordered locales and monotone maps form a category $\OrdLoc$.
\end{proposition}
\begin{proof}	
	\Cref{corollary:identities are monotone} tells us identity morphisms of locales are monotone, so we are left to show that if $f \colon X \to Y$ and $g \colon Y \to Z$ are monotone locale morphisms, then their composition $g\circ f$ is again monotone.
	This follows easily from $R_{g\circ f}= R_g\circ R_f$ and using the interchange law for horizontal and vertical composition in $\Rel$.
\end{proof}

\Cref{section:ordered topological spaces} showed that an ordered space $S$ has well-behaved cones $U \mapsto \up U$ and $U \mapsto \down U$ on its powerset. These completely determine the order $\leq$ because $x \leq y$ if and only if $y \in \up \{x\}$ if and only if $x \in \down \{y\}$. But this relies heavily on the availability of singleton sets. Can we define a pointless analogue of cones that captures some of the properties of $\Leq$? The following results answer this question positively.

Suppose $X$ is a locale with a preorder $\Leq$ on its frame of opens, not necessarily satisfying axiom~\eqref{axiom:V} of an ordered locale. For $U \in \Opens X$, define \emph{localic cones}:
\begin{align*}
	\Up U & = \bigvee \{ V\in\Opens X \mid U\Leq V\},  \\
	\Down U & = \bigvee \{ W\in\Opens X \mid W \Leq U\}.
\end{align*}

\begin{lemma}\label{lemma:properties of localic cones}
	For a locale $X$ and a preorder $\Leq$ on its frame of opens:
	\begin{enumerate}
		\item[(a)] if $U\Leq V$ then $U\sqleq \Down V$ and $V \sqleq \Up U$;
		\item[(b)] $U\sqleq \Up U$ and $U \sqleq \Down U$.
	\end{enumerate}
	If axiom \eqref{axiom:V} is satisfied, then furthermore:
	\begin{enumerate}
		\item[(c)] $U \Leq \Up U$ and $\Down U \Leq U$; 
		\item[(d)] $\Up \Up U = \Up U$ and $\Down \Down U = \Down U$;
		\item[(e)] if $U\sqleq V$ then $\Up U\sqleq \Up V$ and $\Down U\sqleq \Down V$. 
	\end{enumerate}
\end{lemma}
\begin{proof}
  Points (a) and (c) follow immediately from the definition of localic cones, and point (b) from reflexivity of $\Leq$. 
  For (d): applying (c) twice gives ${U \Leq \Up U \Leq \Up \Up U}$, and then applying (a) gives $\Up \Up U \sqleq \Up U$; a similar argument holds for localic downsets. Lastly, for (e), if $U\sqleq V$ and $U\Leq W$ then from \cref{lem:axiomL}(a) we get $V\Leq V'$ for some $W\sqleq V'$, and hence $W\sqleq \Up V$. Thus we get $\Up U\sqleq \Up V$; again, a similar argument holds for localic downsets.
\end{proof}

In direct analogy to \cref{proposition:monotone iff commutes with cones}, we can now determine the monotonicity of morphisms using the localic cones, justifying \cref{definition:morphisms of ordered locales}.

\begin{proposition}\label{proposition:monotone iff commutes with localic cones}
	A locale morphism $f \colon X \to Y$ between ordered locales is monotone if and only if for all $V\in\Opens Y$:
	\[
		\Up f^{-1}(V)\sqleq f^{-1}\left(\Up V\right)
		\quad\text{and}\quad
		\Down f^{-1}(V)\sqleq f^{-1}\left(\Down V\right).
	\]
\end{proposition}
\begin{proof}
			Assume first that $f$ is upper monotone. Using \cref{lemma:properties of localic cones}(c) we get the relation $f^{-1}(V)\Leq \Up f^{-1}(V)$, and $(f^{-1}(V),V)\in R_f$ trivially holds, so there exists $V'\in \Opens Y$ such that $V\Leq V'$ and $(\Up f^{-1}(V),V')\in R_f$. Using \cref{lemma:properties of localic cones}(a) to unpack this, we get 
				\begin{equation*}
					\Up f^{-1}(V)\sqleq f^{-1}(V')\sqleq f^{-1}\left(\Up V\right).
				\end{equation*}
						
			Conversely, suppose that the stated inclusion holds. We show that $f$ is upper monotone, for which we use the characterisation in \cref{lemma:monotonicity unpacked}. So take $U\Leq U'$ and $(U,V)\in R_f$. Then $V\Leq \Up V$ by \cref{lemma:properties of localic cones}(c), and further
				\[
					U'\sqleq \Up U \sqleq \Up f^{-1}(V) \sqleq f^{-1}\left(\Up V\right)
				\]
				follows by \cref{lemma:properties of localic cones}(a) and (e), together with the assumption, giving ${(U',\Up V')\in R_f}$, as desired. The proof for lower monotonicity is analogous.
\end{proof}

	
		


\begin{example}\label{example:ordered sublocales}
	If $j\colon Y \rightarrowtail X$ is a \emph{sublocale}~\cite[Section IX.4]{maclane1994sheaves} of an ordered locale $X$, then $Y$ gets the structure of an ordered locale by endowing $\Opens Y$ with the largest preorder $\Leq_j$ that makes $j$ monotone. Using~\eqref{axiom:V}, we can explicitly write:
	\[
			A\Leq_j B 
			\quad\iff\quad
			\forall U \in R_j(A) \colon \Up U \in R_j(B)
			\;\;\text{ and }\;\;
			\forall V \in R_j(B) \colon \Down V \in R_j(A).
	\]
	Here $R_j(A)=\lbrace U\in \Opens X:(A,U)\in R_j\rbrace$ is the image of $A$ under the relation $R_j$. The preorder $\Leq_j$ inherits axiom~\eqref{axiom:V} from $X$.
	 
	If the frame map $j^{-1}$ admits a left adjoint $j_!\dashv j^{-1}$, \emph{e.g.} when $j$ is \emph{open}, then:
	\[
	 	A\Leq_j B \text{ in $Y$}
	 	\qquad\iff\qquad
	 	j_! A\Leq j_! B \text{ in $X$.}
	\]
\end{example}
 
\begin{example}
	The general construction of~\cref{example:ordered sublocales} generates many specific examples of ordered locales that have no points. For instance, take any ordered space $S$ that is Hausdorff and has no \emph{isolated points}, that is, none of whose singletons are open. In that case the \emph{double negation} (Boolean) sublocale $(\Opens S)_{\neg\neg}$ of $\Opens S$ has no points, but it becomes an ordered locale with the preorder inherited from $S$. 
\end{example}

\section{Opens}\label{sec:opens}

This section discusses turning an ordered space into an ordered locale by taking its opens. 
Given a topological space $S$, the locale $\Opens S$ is defined to be the underlying frame of open subsets of $S$. 
Given a continuous function $g \colon S\to T$, a locale morphism $\Opens g \colon \Opens S \to \Opens T$, that is, a frame morphism $\Opens T\to \Opens S$, is defined by $\Opens g (V) = g^{-1}(V) \in \Opens S $ for $V \in \Opens T$.
This constitutes the functor:
\[
	\Opens \colon \Top \longrightarrow \Loc.
\]

We investigate how this functor interacts with order structures on spaces. First, we extend it on objects. As discussed at the start of \cref{section:ordered locales}, we consider three different ways of doing this, which we recall here.

\begin{definition}\label{construction:order on Loc}
	For open subsets $U$ and $V$ of an ordered space $S$, define:
		\begin{itemize}
			\item the \emph{upper order:} $U\LeqU V$ if and only if $V\subseteq \up U$;
			\item the \emph{lower order:} $U\LeqL V$ if and only if $U\subseteq \down V$;
			\item the \emph{Egli-Milner order}: $U\LeqEM V$ if and only if $U\LeqU V$ and $U\LeqL V$.
		\end{itemize}
	For reference, see~\cite[Section~11.1]{vickers1996topology}. \cref{figure:illustration of causal order on opens} gives visual intuition.

\end{definition}

\begin{lemma}\label{lemma:ordered spaces are ordered locales}
  If $(S,\leq)$ is an ordered topological space, then $\Opens S$ is an ordered locale under $\LeqU$, $\LeqL$, and $\LeqEM$.
\end{lemma}
\begin{proof}
  That $\LeqU$, $\LeqU$, and $\LeqEM$ are preorders on $\Opens S$ follows from \cref{lem:cones}. 
  To verify axiom~\eqref{axiom:V}, suppose that $U_i \LeqL V_i$ for all $i$ ranging over some index set. That means $U_i \subseteq \down V_i$ for all indices $i$. For a fixed index $j$, then $U_j \subseteq \down V_j \subseteq \down \bigcup V_i = \bigcup \down V_i$ by \cref{lem:cones}(d). Therefore $\bigcup U_i \subseteq \down \bigcup V_i$, giving $\bigcup U_i \LeqL \bigcup V_i$. The proof for $\LeqU$ is similar, and it follows directly that $\LeqEM$ must also satisfy~\eqref{axiom:V}.
\end{proof}

This will allow us to define three types of functors $\Opens \colon \OrdTop \to \OrdLoc$, for which the above describes the object component. Before moving on to morphisms, let us look at some more detail at the image of this functor. In the following, the preorder $\Leq$ on the opens of a topological space will always assumed to be the Egli-Milner order $\LeqEM$. The others will be discussed further in \cref{section:upper-lower adjunctions}.

\begin{lemma}\label{lem:Leqfororderedspaces}
  For an open subset $U$ of an ordered space $S$:
  \[
    \Up U = (\up U)\interior
    \qquad \text{ and } \qquad
    \Down U = (\down U)\interior.
  \]
  Therefore, in $\Opens S$:
  \[
    U \Leq V
    \qquad \iff \qquad
    U \subseteq \Down V 
    \text{ and }
    V \subseteq \Up U.
  \]
\end{lemma}
\begin{proof}
  First, we show that:
  \[
		(\down U)\interior \Leq U
		\qquad\text{ and }\qquad
		U \Leq (\up U)\interior.
  \]
	By \cref{construction:order on Loc}, $U \Leq (\up U)\interior$ if and only if $(\up U)\interior \subseteq \up U$, which is vacuously true, and $U\subseteq \down (\up U)\interior$. The latter follows from two applications of \Cref{lem:cones}(a). Similarly $(\down U)\interior \Leq U$.

  Next,	we prove that $\Up U = (\up U)\interior$; the downset-version is similar.
  First, take an element $x\in \Up U$, meaning there exists an open neighbourhood $W\in \Opens S$ of $x$ such that $U\Leq W$. By \cref{construction:order on Loc} this just means that $x\in W\subseteq (\up U)\interior$, and so $\Up U \subseteq (\up U)\interior$.
	Conversely, take an element $x\in (\up U)\interior$. Now $U \Leq (\up U)\interior$, from which it is immediately clear that $x\in \Up U$. Thus $\Up U = (\up U)\interior$

  The stated characterisation of $U \Leq V$ follows immediately.
\end{proof}

Now consider morphisms. The following proposition shows that the open cone condition is precisely what is needed to make the functor $\Opens$ well defined on morphisms. 


\begin{proposition}\label{proposition:open cones iff maps into codomain are monotone}
  The following are equivalent for an ordered space $T$:
	\begin{enumerate}
		\item[(a)] $T$ has open cones;
		\item[(b)] $\Opens g$ is monotone for any continuous monotone function $g \colon S \to T$;
		\item[(c)] $\Opens g$ is monotone for any monotone $g \colon \{ 0<1 \} \to T$.
	\end{enumerate}
\end{proposition}		
\begin{proof}
  To see that (a) implies (b), we need to show that
	\begin{itemize}
	  \item if $U \Leq U'$ and $U' \sqsubseteq g^{-1} V'$, then $V \Leq V'$ and $U \sqsubseteq g^{-1} V$ for some $V$;
	  \item if $U \Leq U'$ and $U \sqsubseteq g^{-1} V$, then $V \leq V'$ and $U' \sqsubseteq g^{-1} V'$ for some $V'$;
	\end{itemize}
	where $U,U'$ are open in $S$ and $V,V'$ are open in $T$.	
	We will pove the first point, as the second is analogous.
	Take $V = (\down V')\interior$.
	Then $V \Leq V'$ by (the proof of) \Cref{lem:Leqfororderedspaces}.
	Also:
	\[
		U
		\subseteq (\down U')\interior
		\subseteq ( \down g^{-1}V')\interior
		\subseteq \down g^{-1}(V')
		\subseteq g^{-1}( \down V')
		= g^{-1}( (\down V')\interior )
		= g^{-1} V
	\]
	where the first inclusion follows from $U \Leq U'$, the second from $U' \subseteq g^{-1} V'$ and \cref{lem:cones}, the third one by definition of interiors, the fourth from the fact that $g$ is monotone and \cref{proposition:monotone iff commutes with cones}, the fifth from the assumption that $T$ has open cones, and the final equality by our choice of $V$.

  Trivially (b) implies (c).

  Finally, we show that (c) implies (a).
  Take an open subset $U$ of $T$. We need to show that $\up U \subseteq (\up U)\interior$. Let $y \in \up U$, meaning there exists $x\in U$ with $x \leq y$. Define a function $g \colon \{0<1\} \to T$ by $g(0)=x$ and $g(1)=y$; we endow the domain of $g$ with the discrete topology, so $g$ is clearly monotone and continuous. By assumption, now $\Opens g$ is a monotone map of ordered locales. In particular, (the proof of) \Cref{proposition:identities and axioms Lpm} now gives 
  an open neighbourhood $U \in\Opens T$ of $y$ such that $U \Leq V$. Therefore $y\in V \subseteq \up U$, and since $V$ is open this means precisely that $y\in (\up U)\interior$. Downsets are similarly open, so $T$ has open cones.
\end{proof}

\begin{corollary}\label{corollary:Loc is well defined}
	There is a functor $\Opens \colon \OrdTop_\OC \to \OrdLoc$ defined on objects by $(S,\leq) \mapsto (\Opens S,\Leq)$ and on morphisms by $g \mapsto \Opens g$.
	\qed
\end{corollary}

\section{Points}\label{sec:points}

This section turns an ordered locale into an ordered space by taking its points. 
A \emph{point} in a locale $X$ is a completely prime filter, that is, a subset $\mathcal{F}\subseteq \Opens X$ satisfying the following properties: 
\begin{itemize}
	\item proper: $\mathcal{F}$ is nonempty, and the least element $0$ of $\Opens X$ is not in $\mathcal{F}$;
  \item upward closed: if $U \in \mathcal{F}$ and $U \sqleq V$, then also $V \in \mathcal{F}$; 
  \item downward directed: if $U,V \in \mathcal{F}$, then also $U \wedge V \in \mathcal{F}$;
  \item completely prime: if $\bigvee U_i \in \mathcal{F}$, then $U_i \in \mathcal{F}$ for some index $i$.
\end{itemize}
Write $\pt(X)$ for the set of points of a locale $X$. It becomes a topological space whose open sets are $\pt(U):=\{ \mathcal{F} \in \pt(X) \mid U \in \mathcal{F} \}$ for $U \in \Opens X$.
If $f \colon X \to Y$ is a locale morphism, and $\mathcal{F}$ is a point of $X$, then $\{V \in \Opens Y \mid f^{-1}(V) \in \mathcal{F}\}$ is a point of $Y$, and this defines a continuous function $\pt(f) \colon \pt(X) \to \pt(Y)$. 
Thus there is a functor $\pt \colon \Loc \to \Top$~\cite[Section~XI.3]{maclane1994sheaves}. In this section we extend it to the ordered setting, starting with objects.

\begin{definition}\label{construction:order on pt}
  For a locale $X$, a preorder $\Leq$ on its frame of opens, and points $\mathcal{F}$ and $\mathcal{G}$, define:
  \[
    \mathcal{F} \leq \mathcal{G}
    \quad \iff \quad
			\forall U \in \mathcal{F}\; \exists V \in \mathcal{G} \colon U \Leq V \;\;\text{ and }\;\;
			\forall V \in \mathcal{G}\; \exists U \in \mathcal{F} \colon U \Leq V    	.
  \]
\end{definition}

It is easy to see that $\leq$ is a preorder on $\pt(X)$ because $\Leq$ is a preorder on $\Opens X$.
Intuitively, no matter how small an open neighbourhood we pick around the (imaginary) point, there is always a neighbourhood around the other point that precedes it. The following lemma gives a simpler characterisation of $\leq$ if $X$ is an ordered locale; compare to \cref{proposition:open cones iff internal cones}.

\begin{lemma}\label{lemma:order on pt in terms of localic cones}
  For points $\mathcal{F}$ and $\mathcal{G}$ in an ordered locale:
  \[
  	\mathcal{F} \leq \mathcal{G}
  	\quad \iff \quad
  	\forall U \in \mathcal{F} \colon \Up U \in \mathcal{G}
  	\;\;\text{ and }\;\;
  	\forall V \in \mathcal{G} \colon \Down V \in \mathcal{F}.
  \]
\end{lemma}
\begin{proof}
  If $\mathcal{F} \leq \mathcal{G}$ then for any $U\in \mathcal{F}$ we can find $V\in\mathcal{G}$ such that $U\Leq V$, which with \cref{lemma:properties of localic cones}(a) gives $V\sqleq \Up U$. Since $\mathcal{G}$ is upwards closed, this implies $\Up U\in\mathcal{G}$. Similarly $\Down V\in\mathcal{F}$ for $V\in\mathcal{G}$. The converse follows from \cref{lemma:properties of localic cones}(b) and (c).
\end{proof}

Next we move to morphisms.

\begin{lemma}\label{proposition:pt is well defined on maps}
  If $f \colon X \to Y$ is a morphism of ordered locales, then the continuous function $\pt(f) \colon \pt(X) \to \pt(Y)$ is monotone.
\end{lemma}
\begin{proof}
  Suppose that $\mathcal{F} \leq \mathcal{G}$ for points in $X$.
  We will show that $\pt(f)(\mathcal{F}) \leq \pt(f)(\mathcal{G})$ in $Y$.
  Let $V \in \pt(f)(\mathcal{F})$, that is, $f^{-1} V \in \mathcal{F}$.
  \Cref{construction:order on pt} then gives a $U \in \mathcal{G}$ with $f^{-1}V \Leq U$. 
  Because $f$ is monotone, there now exists a $W \in \Opens Y$ such that $V \Leq W$ and $U \sqleq f^{-1} W$. 
  As $\mathcal{G}$ is upwards closed, $W \in \pt(f)(\mathcal{G})$. The symmetric condition follows analogously, thus $\pt(f)(\mathcal{F}) \leq \pt(f)(\mathcal{G})$. 
\end{proof}

\begin{corollary}\label{corollary:pt is well defined}
	There is a functor $\pt \colon \OrdLoc \to \OrdTop$ defined on objects by $(X,\Leq) \mapsto (\pt(X),\leq)$ and on morphisms by $f \mapsto \pt(f)$.
	\qed
\end{corollary}

We would like the functor $\pt \colon \OrdLoc \to \OrdTop$ to land in $\OrdTop_\OC$. That is: when does $\pt(X)$ have open cones? To force this we will make an additional (but reasonable) assumption.

\begin{lemma}\label{lemma:ordered locale with axiom pt has open cones}
  If $X$ is an ordered locale and
  \begin{equation}
    U \Leq V \text{ in }X 
    \qquad \implies \qquad
    \pt(U) \Leq \pt(V) \text{ in } \Opens(\pt(X))
    \tag{P}\label{axiom:pt}
  \end{equation}
  then the ordered space $\pt(X)$ has open cones.
\end{lemma}	

\begin{proof}
  Assume~\eqref{axiom:pt}; we will show that $\pt(X)$ has open cones.
	By \Cref{proposition:open cones iff maps into codomain are monotone}, it suffices to show that every monotone function $g \colon \{0<1\} \to \pt(X)$ induces a monotone map $\Opens g$ of locales. 
	That is, we need to show that if $V \in \Opens X$ satisfies $\{1\} \sqleq \pt(V)$, then there exists $U \in \Opens X$ such that $\{0\} \sqleq \pt(U)$ and $\pt(U) \Leq \pt(V)$; the required symmetric condition is analogous.
	Unpacking this further, with $\mathcal{F}=g(0) \leq g(1)=\mathcal{G}$, this means that for all $U \in \mathcal{F}$ there exists $V \in \mathcal{G}$ with $\pt(U)\Leq \pt(V)$, and for every $V \in \mathcal{G}$ there exists $U \in \mathcal{F}$ with $\pt(U)\Leq \pt(V)$. But this follows directly from \Cref{construction:order on pt} and~\eqref{axiom:pt}.
\end{proof}

Intuitively, condition~\eqref{axiom:pt} guarantees that the topology of the locale is rich enough to carry its points into the future and past.
This statement is made more precise in the following lemma, which furthermore provides an alternative proof of \cref{lemma:ordered locale with axiom pt has open cones}.
\begin{lemma}
	An ordered locale $X$ satisfies axiom~\eqref{axiom:pt} if and only if for all $U\in\Opens X$:
	\[
	\up\pt(U) = \pt(\Up U)\qquad\text{and}\qquad \down\pt(U)=\pt(\Down U).
	\]
\end{lemma}
\begin{proof}
	For this proof, we borrow the independent result of \cref{lemma:cone in pt is contained in localic cone} from below, which gives inclusions $\up \pt(U)\subseteq \pt(\Up U)$ and $\down \pt(U)\subseteq \pt(\Down U)$.
	
	Suppose first that~\eqref{axiom:pt} holds. By \cref{lemma:properties of localic cones}(c) we get $U\Leq\Up U$, and hence $\pt(U)\Leq \pt(\Up U)$, which unpacks to give $\pt(\Up U)\subseteq \up\pt(U)$. Equality hence follows using \cref{lemma:cone in pt is contained in localic cone}. The argument for downsets is analogous. 
	
	Now suppose the stated equations hold. Using \cref{lemma:properties of localic cones}(a), if $U\Leq V$ we get $U\sqleq \Down V$ and $V\sqleq \Up U$. These inclusions are respected by points, so we get $\pt(U)\subseteq \pt(\Down V)=\down \pt(V)$ and $\pt(V)\subseteq \pt(\Up U)=\up \pt(U)$, which is precisely what it means for $\pt(U)\Leq \pt(V)$ to hold. This shows~\eqref{axiom:pt} holds.
\end{proof}

In particular, it implies that if $U\Leq V$ then $\pt(U)=\varnothing$ if and only if $\pt(V)=\varnothing$. Note that any locale without points trivially satisfies~\eqref{axiom:pt}. We discuss the axiom briefly further in \cref{sec:conclusion}.

There is a well-known notion of a locale $X$ having \emph{enough points}, which intuitively says that its points distinguish its opens, and formally that the counit $\varepsilon_X$ is an isomorphism. Less well known is the following condition on a \emph{space}: a topological space $S$ has \emph{enough points} if the unit $\eta_S \colon S\to \pt(\Opens(S))$ is surjective. That is, any completely prime filter on a space $S$ with enough points is of the form $\mathcal{F}_x$ for some $x\in S$. Having enough points is  halfway to being sober: a topological space is \emph{sober} if and only if it is $T_0$ and has enough points. Equivalently, a space has enough points if and only if its $T_0$-quotient is sober~\cite{bartels2011soberExceptNotT0}. Ordered spaces with open cones that have enough points always satisfy~\eqref{axiom:pt}.

\begin{lemma}\label{lemma:space with enough points satisfies pt axiom}
	If $S$ is an ordered space with open cones that has enough points,
	then the ordered locale $\Opens S$ satisfies~\eqref{axiom:pt}.
\end{lemma}			
\begin{proof}
  Let $U$ and $V$ be opens in the space $S$ satisfying $U \Leq V$. 
  We need to show that there are inclusion $\pt(U)\subseteq \down \pt(V)$ and $\pt(V)\subseteq \up \pt(U)$; we will prove the first inclusion, the second one is similar.
  Take a point $\mathcal{F}\in \pt(U)$. Since $S$ has enough points, there is an element $x\in U$ such that $\mathcal{F}=\mathcal{F}_x$. Because $U \Leq V$ now $x\in U\subseteq \down V$, so $x \leq y$  for some $y\in V$. Using the fact that $S$ has open cones, \Cref{proposition:unit is monotone} gives that $\mathcal{F}_x\leq \mathcal{F}_y$, and hence $\mathcal{F}=\mathcal{F}_x\in \down \pt(V)$.
\end{proof}

\begin{example}\label{example:ordered space without enough points}
	Any topological space $S$ that does not have enough points becomes an ordered space with open cones by taking the preorder to be equality of elements (\cref{example:open cones with equality}). The induced order on the locale $\Opens S$ and on the space $\pt(\Opens S)$ are also equality. It is then easily seen that axiom~\eqref{axiom:pt} is satisfied, even though $S$ may not have enough points.
%
\end{example}

\section{The adjunction}\label{sec:adjunction}

This section puts everything together, proving that the functors of the previous two sections form an adjunction between ordered spaces and ordered locales, and investigating the resulting equivalence. We build on the fundamental adjunction between spaces and locales~\cite{johnstone1982StoneSpaces,maclane1994sheaves}:
\[
	\begin{tikzcd}[column sep = 1.5cm]
		\Top \arrow[r, "\Opens"{name= X}, shift left=2] & \Loc \arrow[l, "\pt"{name=Y}, shift left=2]
		\ar[phantom,from=X, to=Y,  "\vdash" rotate=90]
	\end{tikzcd}
\]
and will extend it to the ordered setting.
The unit of this classic adjunction is the natural transformation whose components are continuous functions
\begin{align*}
  \eta_S \colon S & \longrightarrow \pt(\Opens S) \\
  x & \longmapsto \mathcal{F}_x = \{ U \in \Opens S \mid x \in U \}
\end{align*}
for topological spaces $S$.
The counit is the natural transformation $\varepsilon \colon \Opens \circ \pt \Rightarrow \mathrm{Id}_{\Loc}$ whose components are frame morphisms
\begin{align*}
  \varepsilon^{-1}_X \colon \Opens X & \longrightarrow \Opens(\pt(X)) \\
  U & \longmapsto \pt(U)
\end{align*}
for locales $X$.	

We start by investigating how this unit and counit interact with the order structures on $\pt(\Opens S)$ and $\Opens(\pt X)$ given by \Cref{construction:order on Loc,construction:order on pt}. 

\begin{lemma}\label{proposition:unit is monotone}\label{lemma:order of points from elements}
  If $x$ and $y$ are points of an ordered space $S$, then in $\pt(\Opens S)$:
  \[
    \mathcal{F}_x \leq \mathcal{F}_y
    \quad\iff\quad
    \forall U \in \mathcal{F}_x \colon (\up U)\interior \in \mathcal{F}_y
    \;\;\text{ and }\;\;
    \forall V \in \mathcal{F}_y \colon (\down V)\interior \in \mathcal{F}_x.
  \]
  Hence the unit $\eta_S \colon S \to \pt(\Opens S)$ is monotone if and only if $S$ has open cones.
\end{lemma}
\begin{proof}
  The first statement follows immediately from \Cref{lem:Leqfororderedspaces,lemma:order on pt in terms of localic cones}. 
  Using this, the unit being monotone, that is, $x \leq y \implies \mathcal{F}_x \leq \mathcal{F}_y$, is equivalent to the open cone condition by \cref{proposition:open cones iff internal cones}.
\end{proof}

\begin{lemma}\label{proposition:counit is monotone}\label{lemma:cone in pt is contained in localic cone}
  If $X$ is a locale, $\Leq$ a preorder on its frame of opens, and $U \in \Opens X$ an open, then in $\pt(X)$:
  \[
    \up\pt(U) \subseteq \pt(\Up U)
    \qquad\text{and}\qquad
    \down\pt(U) \subseteq \pt(\Down U).
  \]
  If $X$ is an ordered locale, then the counit $\varepsilon_X \colon \Opens(\pt X) \to X$ is monotone.
\end{lemma}
\begin{proof}
	For the first statement, take $\mathcal{G}\in \up \pt(U)$, so that there exists $\mathcal{F}\in \pt(U)$ with $\mathcal{F}\leq\mathcal{G}$. In particular, there then exists $V\in\mathcal{G}$ such that $U\Leq V$. \cref{lemma:properties of localic cones}(a) now gives $V\sqleq \Up U$, and since $\mathcal{G}$ is upwards closed this implies $\Up U\in\mathcal{G}$, that is, $\mathcal{G}\in\pt(\Up U)$. The proof for downsets is similar.

	That the counit $\epsilon_X$ is monotone for an ordered locale now follows immediately from the characterisation of monotonicity in \cref{proposition:monotone iff commutes with localic cones}.
\end{proof}

We have arrived at our main result.

\begin{theorem}\label{thm:adjunction}
	There is an adjunction:
	\[
		\begin{tikzcd}[column sep = 1.5cm]
			\OrdTop_{\OC}^\bullet \arrow[r, "\Opens"{name= X}, shift left=2] & \OrdLoc^\bullet
			\arrow[l, "\pt"{name=Y}, shift left=2]
			\ar[phantom,from=X, to=Y,  "\vdash" rotate=90]
		\end{tikzcd}
	\]
	between the full subcategory of $\OrdTop$ of ordered spaces with open cones and enough points, and the full subcategory of $\OrdLoc$ of ordered locales satisfying~\eqref{axiom:pt}.
\end{theorem}
\begin{proof}
  Combine \Cref{corollary:Loc is well defined,corollary:pt is well defined,lemma:space with enough points satisfies pt axiom,proposition:unit is monotone,proposition:counit is monotone}.
\end{proof}

Finally, we investigate the fixed points of this adjunction. In the unordered setting, the fixed points of the classic adjunction between spaces and locales are precisely the \emph{sober} spaces and \emph{spatial} locales~\cite[Section IX.3]{maclane1994sheaves}. 
On the locale side this remains the same in the ordered setting.

\begin{lemma}\label{lem:counitiso}
  If an ordered locale $X$ is spatial, the inverse $(\varepsilon_X)^{-1} \colon X \to \Opens(\pt X)$ of the counit locale morphism is monotone. 
\end{lemma}
\begin{proof}
  This follows immediately from \Cref{lem:axiomL}.
\end{proof}

But the ordered spaces fixed by the adjunction gain an interesting new condition. Recall from \Cref{example:open cones with equality,example:no open cones} that there is no relation between the open cone condition and separation axioms. 

\begin{definition}\label{definition:T0-ordered}
	An ordered space is \emph{$T_0$-ordered} if, whenever $x \not\leq y$, either there exists an open neighbourhood $U$ of $x$ such that $y \notin \up U$, or there exists an open neighbourhood $V$ of $y$ such that $x \notin \down V$.
\end{definition}

The previous definition is a natural weakening of the $T_1$-order separation axiom from~\cite{mccartan1968SeparationAxiomsTopological}; see also~\cite{kunzi2005TiorderedReflections}. Indeed, it is easy to see that being $T_1$-ordered implies being $T_0$-ordered. Also, $T_0$-ordered implies $T_0$. In particular, any pospace is $T_0$-ordered.

\begin{example}\label{example:spacetimes not T0}
	In contrast to \cref{corollary:smooth spacetimes have OC}, not every smooth spacetime is $T_0$-ordered: Minkowski space with one point removed fails this property. This shows that, in general, despite spacetimes being sober, the ordered locale induced by a spacetime loses some information about the causal relations between the original points. Any \emph{causally simple} spacetime \cite[Definition 4.112]{minguzzi2019LorentzianCausalityTheory} is by definition $T_1$-ordered, and hence $T_0$-ordered. In particular, any \emph{globally hyperbolic} spacetime is $T_0$-ordered. 
\end{example}

\begin{lemma}\label{lem:unitiso}
	Let $S$ be an ordered space with open cones that is sober. The inverse $\eta_S^{-1} \colon \pt(\Opens S)\to S$ of the unit is monotone if and only if $S$ is $T_0$-ordered.
\end{lemma}		
\begin{proof}
	Since $\eta_S^{-1}(\mathcal{F}_x)=x$, it follows that $\eta_S^{-1}$ is monotone if and only if $x \not\leq y$ implies $\mathcal{F}_x \not\leq \mathcal{F}_y$. Using \Cref{lemma:order of points from elements} to unpack the order on $\pt(\Opens S)$, together with the open cone condition, this is equivalent to $S$ being $T_0$-ordered.
\end{proof}

\begin{lemma}\label{lemma:pt(X) is T0-ordered}
	If $X$ is an ordered locale, then $\pt(X)$ is $T_0$-ordered.
\end{lemma}
\begin{proof}
	If $\mathcal{F}\not\leq\mathcal{G}$ then by \Cref{lemma:order on pt in terms of localic cones} there exists $U\in\mathcal{F}$ such that $\Up U\notin\mathcal{G}$, or there exists $V\in\mathcal{G}$ such that $\Down V\notin\mathcal{F}$. Following \Cref{lemma:cone in pt is contained in localic cone}, this implies, respectively, that $\mathcal{G}\notin \up \pt(U)$ or $\mathcal{F}\notin \down \pt(V)$, as desired.
\end{proof}

This brings us to the ordered version of Stone duality~\cite[Corollary IX.3.4]{maclane1994sheaves}.

\begin{theorem}\label{theorem:stone duality}
	The adjunction of \Cref{thm:adjunction} restricts to an equivalence of categories between the full subcategory of $\OrdLoc$ of spatial ordered locales satisfying~\eqref{axiom:pt} and the full subcategory of $\OrdTop_\OC$ of sober $T_0$-ordered spaces.
	\qed
\end{theorem}

\section{Adjunctions with upper- and lower orders}\label{section:upper-lower adjunctions}
All proofs above can be separated into analogous upper and lower parts. With this in mind, and without having to do extra work, we get two similar adjunctions that use only the upper or lower order instead of the Egli-Milner order. First, we define two new categories:
	\begin{itemize}
		\item the category of ordered locales with upper monotone maps: $\OrdLoc^{\Up}$;
		\item the category of ordered locales with lower monotone maps: $\OrdLoc^{\Down}$.
	\end{itemize}

Inspecting the proof of \cref{proposition:open cones iff maps into codomain are monotone}, the upper/lower open cone condition corresponds precisely to $\Opens g$ being upper/lower monotone for every continuous monotone function $g$. Hence \Cref{corollary:Loc is well defined} can be rewritten as follows.

\begin{corollary}
	There are functors:
	\begin{center}
		\begin{minipage}{.45\textwidth}
			\begin{align*}
				\OrdTop_\OC^{\up} &\xrightarrow{~\Opens^{\up}~} \OrdLoc^{\Up}\\
				(S,\leq) & \longmapsto (\Opens S,\LeqU)\\
				g & \longmapsto \Opens g
			\end{align*}
		\end{minipage}%
		\begin{minipage}{.45\textwidth}
			\begin{align*}
				\OrdTop_\OC^{\down} &\xrightarrow{~\Opens^{\down}~} \OrdLoc^{\Down}\\
				(S,\leq) & \longmapsto (\Opens S,\LeqL)\\
				g & \longmapsto \Opens g
			\end{align*}
		\end{minipage}
	\end{center}
\end{corollary}


Similarly to how $\LeqEM$ is defined in terms of $\LeqU$ and $\LeqL$, there are two new orders on the points of an ordered locale:
	\begin{itemize}
		\item $\mathcal{F}\leqU\mathcal{G}$ if and only if $\Up U\in\mathcal{G}$ for all $U\in \mathcal{F}$;
		\item $\mathcal{F}\leqL\mathcal{G}$ if and only if $\Down V\in\mathcal{F}$ for all $V\in\mathcal{G}$.
	\end{itemize}
By \cref{lemma:order on pt in terms of localic cones}, the order on the points of an ordered locale from \cref{construction:order on pt} then simply reads $\mathcal{F}\leq\mathcal{G}$ if and only if $\mathcal{F}\leqU\mathcal{G}$ and $\mathcal{F}\leqL\mathcal{G}$.

\begin{example}\label{example:ordered locale with inclusion}
	Recall from \cref{example:locale with inclusion} that any locale $X$ induces an ordered locale $(X,\sqleq)$, where the order is the intrinsic inclusion relation on the opens. For any $U\in\Opens X$ we find $\Up U = X$, the largest open, and $\Down U = U$. Hence the induced orders on the points become:
	\begin{itemize}
		\item $\mathcal{F}\leqU \mathcal{G}$ for all $\mathcal{F},\mathcal{G}\in\pt(X)$;
		\item $\mathcal{F}\leqL\mathcal{G}$ if and only if $\mathcal{G}\subseteq\mathcal{F}$.
	\end{itemize}
	In this way, if $X=\Opens S$ for some topological space $S$, the (opposite of the) specialisation order is reobtained: $x\in\overline{\{y\}}$ if and only if $\mathcal{F}_x\subseteq \mathcal{F}_y$ if and only if $\mathcal{F}_y\leqL \mathcal{F}_x$. 
	
	Further, $X\mapsto (X,\sqleq)$ defines a left adjoint to the forgetful functor:
	\[
	\begin{tikzcd}[column sep = 1.5cm]
		\Loc \arrow[r, ""{name= X}, shift left=2] & \OrdLoc^{\Down}. \arrow[l, ""{name=Y}, shift left=2]
		\ar[phantom,from=X, to=Y,  "\vdash" rotate=90]
	\end{tikzcd}
	\]
	(We actually get a functor $\Loc\to \OrdLoc$, but the adjunction only works when restricted to lower monotone maps.)
\end{example}

\begin{example}
	Starting with an ordered space $S$, we can first take the upper-ordered locale $(\Opens S,\LeqU)$. Using the same proof as \cref{lem:Leqfororderedspaces}, it is easy to verify that $\Up U = (\up U)\interior$. On the other hand, we find $\Down U = S$, since $W\LeqU U$ holds in particular for every open subset $W\supseteq U$. Hence applying the functor $\pt^{\down}$ to $(\Opens S,\LeqU)$ gives the space $\pt(\Opens S)$ that is equipped with the following order: $\mathcal{F}\leqL\mathcal{G}$ if and only if for all $V\in\mathcal{G}$ we have $S=\Down V \in \mathcal{F}$, which is always true. Therefore the induced lower order on the points becomes trivial. Similarly, the resulting space under $\pt^{\up}\circ \Opens^{\down}$ will have the trivial order. 
\end{example}

Looking at the proof of \cref{proposition:pt is well defined on maps}, we see that if a map of locales $f$ is upper/lower monotone, then $\pt(f)$ is monotone with respect to the upper/lower order on points. Hence \cref{corollary:pt is well defined} can be rewritten as follows.

\begin{corollary}
	There are functors
	\begin{center}
		\begin{minipage}{.45\textwidth}
			\begin{align*}
				\OrdLoc^{\Up} &\xrightarrow{~\pt^{\up}~} \OrdTop\\
				(X,\Leq) & \longmapsto (\pt(X),\leqU)\\
				f & \longmapsto \pt(f)
			\end{align*}
		\end{minipage}%
		\begin{minipage}{.45\textwidth}
			\begin{align*}
				\OrdLoc^{\Down} &\xrightarrow{~\pt^{\down}~} \OrdTop\\
				(X,\Leq) & \longmapsto (\pt(X),\leqL)\\
				f & \longmapsto \pt(f)
			\end{align*}
		\end{minipage}
	\end{center}
\end{corollary}


The following result generalises \cref{lemma:order of points from elements}.

\begin{lemma}
	For any ordered space $S$:
		\begin{itemize}
			\item $\eta_S\colon S\to \pt^{\up}(\Opens^{\up} S)$ is monotone if and only if $S$ has open upper cones;
			\item $\eta_S\colon S\to \pt^{\down}(\Opens^{\down} S)$ is monotone if and only if $S$ has open lower cones.
		\end{itemize}
\end{lemma}
\begin{theorem}\label{theorem:adjunction upperlower}
	There are adjunctions
	\begin{equation*}
	\begin{tikzcd}[column sep = 1.5cm]
		\OrdTop_{\OC}^{\up\bullet} \arrow[r, "\Opens^{\up}"{name= X}, shift left=2] & \OrdLoc^{\Up\bullet}
		\arrow[l, "\pt^{\up}"{name=Y}, shift left=2]
		\ar[phantom,from=X, to=Y,  "\vdash" rotate=90]
	\end{tikzcd}
	\quad\text{and}\quad
	\begin{tikzcd}[column sep = 1.5cm]
		\OrdTop_{\OC}^{\down\bullet} \arrow[r, "\Opens^{\down}"{name= X}, shift left=2] & \OrdLoc^{\Down\bullet}
		\arrow[l, "\pt^{\down}"{name=Y}, shift left=2]
		\ar[phantom,from=X, to=Y,  "\vdash" rotate=90]
	\end{tikzcd}
	\end{equation*}
	between the full subcategories of $\OrdTop$ of ordered spaces with open upper/lower cones and enough points, and the full subcategories of $\OrdLoc^{\Up}$ and $\OrdLoc^{\Down}$ of ordered locales satisfying, respectively:
  \begin{align}
    U \LeqU V \text{ in }X 
    \qquad \implies \qquad
    \pt(U) \LeqU \pt(V) \text{ in } \Opens{}^\uparrow(\pt^\uparrow(X))
    \tag{\ref{axiom:pt}$^\uparrow$} \\
    U \LeqL V \text{ in }X 
    \qquad \implies \qquad
    \pt(U) \LeqL \pt(V) \text{ in } \Opens{}^\downarrow(\pt^\downarrow(X))
    \tag{\ref{axiom:pt}$^\downarrow$}
  \end{align}
\end{theorem}

Next, we discuss the fixed points of these adjunctions.

\begin{definition}
	An ordered space $(S,\leq)$ is \emph{$T_\mathrm{U}$-ordered} if $x\nleqslant y$ implies there exists an open neighbourhood $U$ of $x$ such that $y\notin \up U$. Similarly, we say the space is \emph{$T_\mathrm{L}$-ordered} if $x\nleqslant y$ implies there exists an open neighbourhood $V$ of $y$ such that $x\notin \down V$.
\end{definition}

Note that the $T_\mathrm{U}$- and $T_\mathrm{L}$-order axioms both imply the $T_0$-ordered axiom.

\begin{lemma}\label{lemma:unit iso upperlower}
	For a sober ordered space $S$:
		\begin{itemize}
			\item $\eta_S^{-1}\colon \pt^{\up}(\Opens^{\up}S)\to S$ is monotone if and only if $S$ is $T_\mathrm{U}$-ordered;
			\item $\eta_S^{-1}\colon \pt^{\down}(\Opens^{\down}S)\to S$ is monotone if and only if $S$ is $T_\mathrm{L}$-ordered.
		\end{itemize}
\end{lemma}
\begin{proof}
	The map $\eta_S^{-1}$ is monotone in the upper setting if and only if $x\nleqslant y$ implies $\mathcal{F}\nleqslant_\textsc{u}\mathcal{F}_y$. The latter holds if and only if there exists $U\in\mathcal{F}_x$ such that $\Up U\notin \mathcal{F}_y$. Since $S$ has open upper cones and by the proof of \cref{lem:Leqfororderedspaces} we see $\Up U = \up U$, so we obtain the equivalence between the $T_\mathrm{U}$-ordered condition. The proof in the lower setting is analogous.
\end{proof}

Generalising the proof of \cref{lemma:pt(X) is T0-ordered} gives the following.

\begin{lemma}
	For any ordered locale $X$, the space $\pt^{\up}(X)$ is $T_\mathrm{U}$-ordered, and the space $\pt^{\down}(X)$ is $T_\mathrm{L}$-ordered.
\end{lemma}

Putting it all together gives two new dualities.

\begin{theorem}\label{theorem:duality upperlower}
	The adjunctions of \cref{theorem:adjunction upperlower} restrict to equivalences between the categories of ordered locales and upper/lower monotone morphisms satisfying the upper/lower version of~\eqref{axiom:pt}, and the categories of sober $T_\mathrm{U}$/$T_\mathrm{L}$-ordered spaces with open upper/lower cones.
\end{theorem}

\section{Relation to Esakia duality}\label{section:esakia duality}

Priestley and Esakia spaces arise as topological representations of distributive lattices \cite{priestley:lattices} and Heyting algebras \cite{esakia1974topologicalKripkemodels}, respectively. This section shows how Esakia duality arises from our adjunctions. We recall basic definitions, but refer to the literature for detail: see \cite{bezhanishvili2010BitopologicalDualityDistributive} for a detailed treatment from a bitopological point of view, \cite{landsman2021LogicQuantumMechanics} for an expository account, and the appendix of \cite{davey2003CoalgebraicViewHeyting} for a detailed proof of Esakia duality.

\begin{definition}
	A \emph{Priestley space} is an ordered space $(S,\leq)$ where $S$ is compact, and the \emph{Priestley separation axiom} holds: if $x\nleqslant y$ there exists a clopen upset $U$ such that $x\in U$ and $y\notin U$. 
\end{definition}

Priestley spaces are also known as \emph{ordered Stone spaces}, since any Priestley space is a Stone space by \emph{e.g.}\ \cite[Lemma~3.2]{bezhanishvili2010BitopologicalDualityDistributive}. Hence, importantly here, every Priestley space is sober. 

\begin{lemma}\label{lemma:priestley spaces are TU and TL ordered}
	Any Priestley space is $T_\mathrm{U}$- and $T_\mathrm{L}$-ordered.
\end{lemma}
\begin{proof}
	If $x\nleqslant y$, then by the Priestley separation axiom there exists a clopen upset $U$ with $x\in U$ but $y\notin U$. Since $U$ is an upset we get $\up U = U$, so this gives the $T_\mathrm{U}$-ordered condition. Further, the complement of $U$ is a clopen downset that contains $y$ but not $x$, giving the $T_\mathrm{L}$-ordered condition.
\end{proof}

Using \cite[Theorem~4.2]{davey2003CoalgebraicViewHeyting} we can adapt the definition of Esakia spaces in terms of open cone conditions as follows.

\begin{definition}\label{definition:esakia spaces}
	A Priestley space is called:
		\begin{itemize}
			\item an \emph{Esakia space} if it has open lower cones;
			\item a \emph{co-Esakia space} if it has open upper cones;
			\item a \emph{bi-Esakia space} if it has open cones.
		\end{itemize}
\end{definition}
The definition of morphisms between Esakia spaces are slightly more involved, essentially due to the fact that they have to ensure the induced map between Heyting algebras respects the Heyting implication. For us, it suffices to know that morphisms between Esakia spaces are, in particular, continuous and monotone. Refer to \cite[Section~7]{bezhanishvili2010BitopologicalDualityDistributive} for more details. Denote by $\Esakia$, $\coEsakia$, and $\biEsakia$ the categories of Esakia-, co-Esakia-, and bi-Esakia spaces, respectively. Similarly, denote by $\Heyt$, $\coHeyt$, and $\biHeyt$ the categories of Heyting-, co-Heyting-, and bi-Heyting algebras, respectively. \emph{Esakia duality} says that there are equivalences of categories:
\[
  \Esakia \cong \Heyt\op
  \qquad\qquad
  \coEsakia \cong \coHeyt\op
  \qquad\qquad
  \biEsakia \cong \biHeyt\op
\]
\Cref{definition:esakia spaces} lets us interpret Esakia spaces as ordered spaces with open cones.

\begin{lemma}\label{lemma:esakia spaces include in ordtop open lower cones}
	There are inclusion functors:
		\[
			\Esakia \hookrightarrow \OrdTop_{\OC}^{\down\bullet}
			\qquad
			\coEsakia \hookrightarrow \OrdTop_{\OC}^{\up\bullet}
			\qquad
			\biEsakia \hookrightarrow \OrdTop_{\OC}^\bullet
		\]
\end{lemma}

Moreover, the following proposition shows that these inclusion functors land exactly in the respective fixed points.

\begin{proposition}\label{proposition:esakia spaces are fixed points}
	The following holds:
		\begin{itemize}
			\item any Esakia space $S$ is isomorphic to $\pt^{\down}(\Opens^{\down} S)$;
			\item any co-Esakia space $S$ is isomorphic to $\pt^{\up}(\Opens^{\up}S)$;
			\item any bi-Esakia space $S$ is isomorphic to $\pt(\Opens S)$.
		\end{itemize}
\end{proposition}
\begin{proof}
	This follows from the fact that Priestley spaces are sober, and by combining \cref{lemma:unit iso upperlower,lemma:priestley spaces are TU and TL ordered}.
\end{proof}


We now exhibit one sense in which Esakia duality fits into our framework. To obtain an Esakia space from a Heyting algebra, one constructs the space of prime filters, suitably topologised. We denote the resulting functor by ${\pf\colon \Heyt\op\to\Esakia}$, and similarly for the other varieties of Esakia spaces. Conversely, the lattice of clopen upsets of an Esakia space defines a Heyting algebra, and we denote the resulting functor by $\Clopup\colon\Esakia\to \Heyt\op$. Consider the following diagrams:
	\[
	\begin{tikzcd}[column sep = 1cm]%
		\Esakia \arrow[d,hookrightarrow]\arrow[r, "\Clopup"{name= A}, shift left=2]
		& 
		\Heyt\op  \arrow[l, "\pf"{name=B}, shift left=2]\arrow[d, "\Opens^{\down}\circ \pf"]
		\\
		\OrdTop_{\OC}^{\down\bullet} \arrow[r, "\Opens^{\down}"{name= X}, shift left=2] 
		&
		\OrdLoc^{\Down\bullet}
		\arrow[l, "\pt^{\down}"{name=Y}, shift left=2]
		\ar[phantom,from=X, to=Y,  "\vdash" rotate=90]
	\end{tikzcd}
\quad
	\begin{tikzcd}[column sep = 1cm]%
		\coEsakia \arrow[d,hookrightarrow]\arrow[r, "\Clopup"{name= A}, shift left=2]
		& 
		\coHeyt\op  \arrow[l, "\pf"{name=B}, shift left=2]\arrow[d, "\Opens^{\up}\circ \pf"]
		\\
		\OrdTop_{\OC}^{\up\bullet} \arrow[r, "\Opens^{\up}"{name= X}, shift left=2] 
		&
		\OrdLoc^{\Up\bullet}
		\arrow[l, "\pt^{\up}"{name=Y}, shift left=2]
		\ar[phantom,from=X, to=Y,  "\vdash" rotate=90]
	\end{tikzcd}
	\]%
	\[
	\begin{tikzcd}[column sep = 1cm]%
		\biEsakia \arrow[d,hookrightarrow]\arrow[r, "\Clopup"{name= A}, shift left=2]
		& 
		\biHeyt\op  \arrow[l, "\pf"{name=B}, shift left=2]\arrow[d, "\Opens\circ \pf"]
		\\
		\OrdTop_{\OC}^{\bullet} \arrow[r, "\Opens"{name= X}, shift left=2] 
		&
		\OrdLoc^{\bullet}
		\arrow[l, "\pt"{name=Y}, shift left=2]
		\ar[phantom,from=X, to=Y,  "\vdash" rotate=90]
	\end{tikzcd}
	\]

Here $\Opens^{\down}\circ\pf\colon\Heyt\op\to\OrdLoc^{\Down\bullet}$ lands in the fixed points of $\OrdLoc^{\Down\bullet}$, since it is in the image of $\Opens^{\down}$. The following shows that the left- and right-directed squares in the above diagrams commute up to isomorphism.
\begin{proposition}
	There are natural isomorphisms:
		\begin{align*}
			\pf & \cong \pt^{\down}\circ \Opens\nolimits^{\down}\circ \pf
			& \Opens\nolimits^{\down}\circ \pf\circ \Clopup &\cong \Opens\nolimits^{\down} \\
			\pf & \cong \pt^{\up}\circ \Opens\nolimits^{\up}\circ \pf
			& \Opens\nolimits^{\up}\circ \pf\circ \Clopup&\cong \Opens\nolimits^{\up} \\
			\pf & \cong \pt\circ \Opens\circ \pf
			& \Opens\circ \pf\circ \Clopup & \cong \Opens
		\end{align*}
\end{proposition}
\begin{proof}
	The latter isomorphisms follow immediately from the fact that $\pf$ and $\Clopup$ are mutually inverse up to isomorphism, and the former from \cref{proposition:esakia spaces are fixed points}.
\end{proof}

\section{Conclusion}\label{sec:conclusion}

To conclude, we raise directions for further research.
\begin{itemize}
  \item Axiom~\eqref{axiom:pt} in \cref{lemma:ordered locale with axiom pt has open cones}~is a somewhat \emph{ad hoc} solution to ensure open cones in $\pt(X)$. This condition, albeit natural, is unsatisfactory from a localic perspective, since it is stated in terms of points, instead of purely in terms of opens. Note that the definition of an ordered locale actually does not axiomatise ordered spaces with open cones, but instead axiomatises \emph{arbitrary} ordered spaces: the open cone assumption is never used in the proof of \cref{lemma:ordered spaces are ordered locales}. There are several natural candidates (cf. \cref{example:failure of Lambda}), but currently we are not aware of an additional axiom for ordered locales, purely stated in terms of opens, that could replace the need for axiom~\eqref{axiom:pt}.
  One may be similarly dissatisfied with the need to restrict to spaces with enough points (\cref{lemma:space with enough points satisfies pt axiom}). In \cref{example:ordered space without enough points} we saw examples of ordered spaces \emph{without} enough points that nevertheless satisfy axiom~\eqref{axiom:pt}. This shows that the enough points assumption is not necessary. Future work could elucidate both of these restrictions and generalise the adjunction.
  
  \item Ordered spaces are an important source of examples in (directed) algebraic topology, the study of the \emph{fundamental category} of a (directed) topological space, which consists of homotopy classes of (directed) paths between points~\cite{grandis2009DirectedAlgebraicTopology}. In the undirected case, there are good notions of fundamental category geared towards locales rather than spaces~\cite{kennison1989WhatFundamentalGroup}. Our results could be used to develop a good notion of fundamental category of an ordered locale.

  \item Ordered spaces generalise to locally ordered spaces, also called \emph{streams}~\cite{krishnan2009ConvenientCategoryLocally}, which arise often in concurrent computation. Our results could similarly be generalised to `localic streams', which entails lifting the adjunction to categories of cosheaves.

  Ordered spaces also generalise to topological categories. Similarly, categorical versions of the Egli-Milner order exist~\cite{lehmann}. Our results could be generalised to \emph{sites}. This also raises the question of what the natural notion of \emph{sheaf} is on an ordered locale. Given such a notion, our framework could be used to study \emph{contextuality} \cite{gogioso2021SheafTheoreticStructureDefinite}.
	
  \item Ordered spaces form the basis of structural approaches to \emph{causality} in quantum gravity~\cite{kronheimer1967structure,christensen2005CausalSitesQuantum}. Our results could form the basis for a pointfree approach to causality in quantum gravity.

  Similarly, ordered spaces can satisfy many properties of the so-called \emph{causal ladder} that make them more well-behaved as models of physical spacetime~\cite{minguzzi2019LorentzianCausalityTheory}. For example, antisymmetry of the preorder relates to the existence of closed causal curves. Our results could form a starting point to find pointfree analogues of such properties.

  \item \emph{Quantales} are a noncommutative generalisation of locales, which are used in the study of \emph{observational logic} \cite{abramsky1993QuantalesObservationalLogic} and \emph{resource theories} \cite{gonda2021ResourceTheoriesQuantaleModules}. Our ordered locales could generalise to quantales, to help study causality in those settings.
  
  \item Locales can be categorified to monoidal categories via \emph{tensor topology}~\cite{enriquemoliner2020TensorTopology,enriqueMoliner2017space}. Our results justify looking for a structure on a monoidal category that decategorifies to an ordered locale. Similarly, relations to \emph{monoidal topology}~\cite{tholen2009OrderedTopologicalStructures,hofmann2014MonoidalTopologyCategorical} and \emph{promonoidal categories} \cite{hefford2022PrePromonoidalStructure} could be investigated. 

  \item Ordered spaces are sometimes also regarded as \emph{bitopological spaces}, as the order induces a second topology. There is a notion of biframe, and a corresponding categorical duality~\cite{schauerte}. The relationship to our results should be investigated, in particular in light of the bitopological treatment of Esakia duality in \cite{bezhanishvili2010BitopologicalDualityDistributive}.
  
  \item A notion of \emph{ordered locale} $X$ has also been defined in \cite{townsend1997LocalicPriestleyDuality} and \cite[Chapter~5]{townsend1996preframeTechniquesLocales} in terms of closed sublocales of $X\times X$, where $X$ is a compact Hausdorff locale. This allows for a constructive treatment of Priestley duality. Besides the restriction to compact Hausdorff locales, this type of order structure is apparently different from our preorder relation $\Leq$ on the frame of opens of a locale. The relation between the two notions could be investigated.
\end{itemize}

\bibliographystyle{plain}
\bibliography{bibliography}

\begin{thebibliography}{10}

\bibitem{abramsky1993QuantalesObservationalLogic}
S.~Abramsky and S.~Vickers.
\newblock Quantales, observational logic and process semantics.
\newblock {\em Mathematical Structures in Computer Science}, 3(2):161--227,
  1993.

\bibitem{bartels2011soberExceptNotT0}
T.~Bartels.
\newblock Sober except not {$T_0$}?
\newblock MathOverflow, 2011.
\newblock \url{https://mathoverflow.net/q/72217}.

\bibitem{bezhanishvili2010BitopologicalDualityDistributive}
G.~Bezhanishvili, N.~Bezhanishvili, D.~Gabelaia, and A.~Kurz.
\newblock Bitopological duality for distributive lattices and {{Heyting}}
  algebras.
\newblock {\em Mathematical Structures in Computer Science}, 20(3):359--393,
  June 2010.

\bibitem{sorkinetal:causalset}
L.~Bombelli, J.~Lee, D.~Meyer, and R.~D. Sorkin.
\newblock Space-time as a causal set.
\newblock {\em Physical Review Letters}, 59:521, 1987.

\bibitem{christensen2005CausalSitesQuantum}
J.D. Christensen and L.~Crane.
\newblock Causal sites as quantum geometry.
\newblock {\em Journal of Mathematical Physics}, 46(12):122502, 2005.

\bibitem{chrusciel2012LorentzianCausalityContinuous}
P.~T. Chruściel and J.~D.~E. Grant.
\newblock On {{Lorentzian}} causality with continuous metrics.
\newblock {\em Classical and Quantum Gravity}, 29(14):145001, 2012.

\bibitem{davey2003CoalgebraicViewHeyting}
B.~A. Davey and J.~C. Galati.
\newblock A {{Coalgebraic View}} of {{Heyting Duality}}.
\newblock {\em Studia Logica}, 75(3):259--270, December 2003.

\bibitem{enriqueMoliner2017space}
P.~{Enrique Moliner}, C.~{Heunen}, and S.~{Tull}.
\newblock {Space in Monoidal Categories}.
\newblock In {\em Quanum Physics and Logic}, volume 266 of {\em Electronic
  Proceedings in Theoretical Computer Science}, pages 399--410, 2017.

\bibitem{enriquemoliner2020TensorTopology}
P.~{Enrique Moliner}, C.~Heunen, and S.~Tull.
\newblock Tensor topology.
\newblock {\em Journal of Pure and Applied Algebra}, 224(10):106378, 2020.

\bibitem{esakia1974topologicalKripkemodels}
L.~L. Esakia.
\newblock Topological {K}ripke models.
\newblock {\em Doklady Akademii Nauk SSSR}, 214(2):298--301, 1974.

\bibitem{escardo:synthetic}
M.~Escard{\'o}.
\newblock {\em Synthetic topology}, volume~87 of {\em Electronic Proceedings in
  Theoretical Computer Science}.
\newblock Elsevier, 2004.

\bibitem{fajstrup2016DirectedAlgebraicTopology}
L.~Fajstrup, E.~Goubault, E.~Haucourt, S.~Mimram, and M.~Raussen.
\newblock {\em Directed {{Algebraic Topology}} and {{Concurrency}}}.
\newblock Springer, 2016.

\bibitem{forrest1996OntologyTopologyTheory}
P.~Forrest.
\newblock From ontology to topology in the theory of regions.
\newblock {\em The Monist}, 79(1):34--50, 1996.

\bibitem{gierzetal:domains}
G.~Gierz, K.~H. Hofmann, K.~Keimel, J.~D. Lawson, M.~W. Mislove, and D.~S.
  Scott.
\newblock {\em Continuous lattices and domains}.
\newblock Cambridge University Press, 2003.

\bibitem{gogioso2021SheafTheoreticStructureDefinite}
S.~Gogioso and N.~Pinzani.
\newblock The sheaf-theoretic structure of definite causality.
\newblock {\em Proceedings 18th International Conference on Quantum Physics and
  Logic}, 2021.

\bibitem{gonda2021ResourceTheoriesQuantaleModules}
T.~Gonda.
\newblock {\em Resource Theories as Quantale Modules}.
\newblock PhD thesis, University of Waterloo, 2021.

\bibitem{grandis2009DirectedAlgebraicTopology}
M.~Grandis.
\newblock {\em Directed Algebraic Topology}.
\newblock {Cambridge University Press}, 2009.

\bibitem{grant2020FutureNotAlwaysOpen}
J.~D.~E. Grant, M.~Kunzinger, C.~Sämann, and R.~Steinbauer.
\newblock The future is not always open.
\newblock {\em Letters in Mathematical Physics}, 110(1):83--103, 2020.

\bibitem{haucourt2009ComparingTopologicalModels}
E.~Haucourt.
\newblock Comparing {{Topological Models}} for {{Concurrency}}.
\newblock {\em Electronic Notes in Theoretical Computer Science}, 230:111--127,
  2009.

\bibitem{hefford2022PrePromonoidalStructure}
J.~Hefford and A.~Kissinger.
\newblock On the pre- and promonoidal structure of spacetime.
\newblock arXiv:2206.09678, 2022.

\bibitem{hofmann2014MonoidalTopologyCategorical}
D.~Hofmann, G.~J. Seal, and W.~Tholen, editors.
\newblock {\em Monoidal topology}.
\newblock {Cambridge University Press}, 2014.

\bibitem{johnstone1982StoneSpaces}
P.~T. Johnstone.
\newblock {\em Stone Spaces}.
\newblock {Cambridge university press}, 1982.

\bibitem{johnstone:point}
P.~T. Johnstone.
\newblock The point of pointless topology.
\newblock {\em Bulletin of the American Mathematical Society}, 8(1):41--53,
  1983.

\bibitem{johnstone2002elephant}
P.~T. Johnstone.
\newblock {\em Sketches of an {{Elephant}}}.
\newblock {Oxford University Press}, 2002.

\bibitem{jonssontarski}
B.~J{\'o}nsson and A.~Tarski.
\newblock Boolean algebras with operators.
\newblock {\em American Journal of Mathematics}, 73(4):891--939, 1951.

\bibitem{kennison1989WhatFundamentalGroup}
J.~F. Kennison.
\newblock What is the fundamental group?
\newblock {\em Journal of Pure and Applied Algebra}, 59(2):187--200, 1989.

\bibitem{krishnan2009ConvenientCategoryLocally}
S.~Krishnan.
\newblock A convenient category of locally preordered spaces.
\newblock {\em Applied Categorical Structures}, 17(5):445--466, 2009.

\bibitem{kronheimer1967structure}
E.~H. Kronheimer and R.~Penrose.
\newblock On the structure of causal spaces.
\newblock {\em Mathematical Proceedings of the Cambridge Philosophical
  Society}, 63(2):481--501, 1967.

\bibitem{kunzi2005TiorderedReflections}
H.-P.~A. Künzi and T.~A. Richmond.
\newblock {$T_i$}-ordered reflections.
\newblock {\em Applied General Topology}, 6(2):207--216, 2005.

\bibitem{landsman2021LogicQuantumMechanics}
N.~P. Landsman.
\newblock The {{Logic}} of {{Quantum Mechanics}} ({{Revisited}}).
\newblock In Gabriel Catren and Mathieu Anel, editors, {\em New {{Spaces}} in
  {{Physics}}: {{Formal}} and {{Conceptual Reflections}}}, volume~2, pages
  85--113. {Cambridge University Press}, {Cambridge}, 2021.

\bibitem{landsman2021GR}
N.P. Landsman.
\newblock {\em Foundations of General Relativity}.
\newblock Radboud University Press, 2021.

\bibitem{lehmann}
D.~J. Lehmann.
\newblock {\em Categories for fixpoint semantics}.
\newblock PhD thesis, University of Warwick, 1976.

\bibitem{maclane1994sheaves}
S.~Mac{}Lane and I.~Moerdijk.
\newblock {\em Sheaves in Geometry and Logic}.
\newblock Springer, 1994.

\bibitem{mccartan1968SeparationAxiomsTopological}
S.~D. McCartan.
\newblock Separation axioms for topological ordered spaces.
\newblock {\em Mathematical Proceedings of the Cambridge Philosophical
  Society}, 64(4):965--973, 1968.

\bibitem{minguzzi2019LorentzianCausalityTheory}
E.~Minguzzi.
\newblock Lorentzian causality theory.
\newblock {\em Living Reviews in Relativity}, 22(1):3, 2019.

\bibitem{nachbin}
L.~Nachbin.
\newblock {\em Topology and order}.
\newblock Van Nostrand, 1965.

\bibitem{panangaden2014CausalityPhysicsComputation}
P.~Panangaden.
\newblock Causality in physics and computation.
\newblock {\em Theoretical Computer Science}, 546:10--16, 2014.

\bibitem{picado2012frames}
J.~Picado and A.~Pultr.
\newblock {\em Frames and Locales}.
\newblock Springer, 2012.

\bibitem{priestley:lattices}
H.~A. Priestley.
\newblock Ordered topological spaces and the representation of distributive
  lattices.
\newblock {\em Proceedings of the London Mathematical Society},
  s3-24(3):507--530, 1972.

\bibitem{richmond2020GeneralTopologyIntroduction}
T.~Richmond.
\newblock {\em General topology}.
\newblock {De Gruyter}, 2020.

\bibitem{schauerte}
A.~Schauerte.
\newblock {\em Biframes}.
\newblock PhD thesis, McMaster University, 1992.

\bibitem{stone1936TheoryRepresentationBoolean}
M.~H. Stone.
\newblock The {{Theory}} of {{Representation}} for {{Boolean Algebras}}.
\newblock {\em Transactions of the American Mathematical Society},
  40(1):37--111, 1936.

\bibitem{tholen2009OrderedTopologicalStructures}
W.~Tholen.
\newblock Ordered topological structures.
\newblock {\em Topology and its Applications}, 156(12):2148--2157, 2009.

\bibitem{townsend1996preframeTechniquesLocales}
C.~Townsend.
\newblock {\em Preframe Techniques in Constructive Locale Theory}.
\newblock PhD thesis, University of London, 1996.

\bibitem{townsend1997LocalicPriestleyDuality}
C.~Townsend.
\newblock Localic {{Priestley}} duality.
\newblock {\em Journal of Pure and Applied Algebra}, 116(1):323--335, March
  1997.

\bibitem{vickers1996topology}
S.~Vickers.
\newblock {\em Topology via Logic}.
\newblock {Cambridge University Press}, 1996.

\bibitem{ward:orderedspaces}
L.~E. {Ward, Jr.}
\newblock Partially ordered topological spaces.
\newblock {\em Proceedings of the American Mathematical Society},
  5(1):144--161, 1954.

\end{thebibliography}

\end{document}